\documentclass[sn-mathphys-num]{sn-jnl}

\usepackage{graphicx}%
\usepackage{multirow}%
\usepackage{amsmath,amssymb,amsfonts}%
\usepackage{amsthm}%
\usepackage{mathrsfs}%
\usepackage[title]{appendix}%
\usepackage{xcolor}%
\usepackage{textcomp}%
\usepackage{manyfoot}%
\usepackage{booktabs}%
\usepackage{listings}%
\usepackage{mathtools}
\usepackage{subcaption}

\theoremstyle{plain}
\newtheorem{proposition}{Proposition}[section]

\newtheorem{theorem}[proposition]{Theorem}
\newtheorem{corollary}[proposition]{Corollary}

\theoremstyle{definition}
\newtheorem{definition}[proposition]{Definition}
\newtheorem{example}[proposition]{Example}

\newtheorem{assumption}[proposition]{Assumption}

\theoremstyle{remark}
\newtheorem{remark}[proposition]{Remark}

\usepackage{lipsum}
\usepackage{epstopdf}
\ifpdf
  \DeclareGraphicsExtensions{.eps,.pdf,.png,.jpg}
\else
  \DeclareGraphicsExtensions{.eps}
\fi

\usepackage{hyperref}
\usepackage{soul}

\usepackage[shortlabels]{enumitem}
\usepackage{resizegather}

\DeclareMathOperator*{\argmin}{\operatorname{argmin}}

\usepackage[ruled,vlined,linesnumbered,algo2e,algosection]{algorithm2e}

\usepackage[capitalize,noabbrev,nameinlink]{cleveref}

\newcommand{\N}{\mathbb{N}}

\newcommand{\shift}{r}

\setlist[enumerate]{label=(\alph*)}
\SetKw{KwOr}{or}

\begin{document}

\title[Fully Convergent PGM with Momentum under Geometric Constraints]{Fully Convergent Projection-based  Methods with Momentum under Nonconvex Geometric Constraints}


\author[1]{\fnm{Matteo} \sur{Lapucci}}\email{matteo.lapucci@unifi.it}

\author*[2]{\fnm{Diego} \sur{Scuppa}}\email{diego.scuppa@uniroma1.it}

\affil[1]{\orgdiv{Dipartimento di Ingegneria dell'Informazione Firenze}, \orgname{Università di Firenze}, \orgaddress{\street{Via di Santa Marta 3}, \city{Firenze}, \postcode{50135}, \state{FI}, \country{Italy}}}

\affil*[2]{\orgdiv{Dipartimento di Ingegneria Informatica, Automatica e Gestionale ``Antonio Ruberti''}, \orgname{Sapienza Università di Roma}, \orgaddress{\street{Via Ariosto 25}, \city{Roma}, \postcode{00185}, \state{RM}, \country{Italy}}}


\abstract{
Nonlinear optimization problems with  complicated, nonconvex, yet geometrically structured constraints can be tackled by projected-gradient methods: under weak regularity assumptions, these approaches were recently proved to possess convergence properties to the strongest stationarity conditions. 
In this work, we show how momentum terms, commonly used in nonlinear optimization to speed up the convergence process, can be integrated within this algorithmic framework without harming convergence guarantees. 

Preliminarily, we highlight an intrinsic issue induced by the direct replacement of the negative gradient with a general descent direction within the projected approach. 
Then, we present suitable backtracking mechanisms for the pre-projection step, allowing us to integrate momentum terms in the direction. By this technique, we can specifically ensure, without any smoothness assumptions, convergence to Mordukhovich stationarity as long as the base directions asymptotically revert to the negative gradient for small stepsizes; moreover, if the base search direction reverts exactly to the negative gradient for the smallest steps, the algorithm is proved to converge to Bouligand and Proximally stationary points, with and without (local) smoothness assumptions respectively.

Finally, the proposed procedure is  numerically tested on some classes of problems, namely, sparsity and bounded-rank constrained problems; the results indicate that the proposed method is computationally effective, taking advantage of the additional information provided by the momentum term.
}



\maketitle

\section{Introduction}

In this paper, we deal with optimization problems of the form
\begin{equation}
    \label{eqn:prob}
    \begin{aligned}
        \min_{x\in \mathbb{X}}\;&f(x)\\\text{s.t. }& x\in D,    
    \end{aligned}
\end{equation}
where $f:\mathbb{X}\to \mathbb{R}$ is a differentiable function over a Euclidean space $\mathbb{X}$ and $D\subset \mathbb{X}$ is a generic, possibly complicated set, only assumed to be nonempty and closed:  it might be nonconvex, nonsmooth, discontinuous or disconnected. However, we assume to have access to a projection operator onto this set.

This very general setting, referred to as optimization problems with structured \textit{geometric constraints}~\cite{jia2023augmented,de2023constrained,kanzow2023inexact}, covers a wide variety of relevant optimization problems: complementarity constraints~\cite{ye1999optimality,pang1999complementarity}, vanishing constraints~\cite{achtziger2008mathematical}, switching constraints~\cite{mehlitz2020stationarity}, cardinality constraints~\cite{beck2013sparsity} and low-rank constraints~\cite{levin2023finding,olikier2026lowrank} are all instances of this general framework.

To address problems of the form~\eqref{eqn:prob}, the natural route consists in exploiting first (and possibly higher) order information available about the objective function and the available projection operator to handle the constraint. Algorithms inspired by the projected gradient method (PGM) for problems with convex constraints (see, e.g.,~\cite[Ch.\ 2]{bertsekas1995nonlinear},~\cite[Ch.\ 20]{grippo2023introduction} or~\cite{birgin2014spectral}) have been devised to deal with specific instances of the problem, such as the Iterative Hard Thresholding method for sparsity constraints~\cite{blumensath2008iterative,beck2013sparsity}. The method was later generalized to the wider scenario~\cite{jia2023augmented}: coupled with a backtracking search along the projection path,
the algorithm then guarantees  the (nonmonotone~\cite{grippo1986nonmonotone,zhang2004nonmonotone}) sufficient decrease of the objective while leveraging a spectral stepsize selection rule 
well-known for the case of convex constraints~\cite{birgin2000nonmonotone} and which was observed to be similarly efficient in the new scenario.

Actually, the PGM represents a specialized instance of the proximal gradient method for composite optimization problems; for the case of a nonconvex nonsmooth term, proximal algorithms with monotone and nonmonotone line searches have been thoroughly analyzed in recent years~\cite{kanzow2022convergence,de2023proximal}, possibly taking into account the use of the Forward-Backward Envelope (FBE)~\cite{patrinos2013proximal}  as a merit function for  a globalization strategy~\cite{stella2017simple,themelis2018forward,de2022proximal}.

For the particular case of problem~\eqref{eqn:prob}, however, stationarity-type necessary optimality conditions are not uniquely defined. In fact, three stationarity conditions, with increasing strength in general, can be stated, each one depending on the choice of the normal cone that we are willing to consider~\cite[Def.\ 6.3 and Ex.\ 6.16]{rockafellar1998variational}; briefly, following the nomenclature used for example in~\cite{olikier2025projected}, we have

\smallskip
\begin{itemize}
    \item \textit{M-stationarity}~\cite{benko2017estimating,kanzow2022convergence},  where ``M'' stands for Mordukhovich, requires the negative gradient of $f$ to belong to the general normal cone $\mathcal{N}_D(x)$ (also known as limiting~\cite{mordukhovich2018variational} or Mordukhovich normal cone);
    \item \textit{B-stationarity}~\cite{pang1999complementarity,benko2017estimating,mordukhovich2018variational}, where ``B'' stands for Bouligand, considers the regular normal cone $\widehat{\mathcal{N}}_D(x)$  (also known as Fréchet normal cone~\cite{mordukhovich2018variational,rockafellar1998variational,pauwels2024note});
    \item \textit{P-stationarity}~\cite{olikier2025projected}, where ``P'' stands for ``proximally'', takes into account the proximal normal cone $\widehat{{\mathcal{N}}}^P_D(x)$. 
\end{itemize}

\smallskip
As pointed out in~\cite{olikier2025projected}, it is possible to show that B- and P-stationarity are indeed the strongest necessary conditions for local optimality with problem~\eqref{eqn:prob} if the objective $f$ is continuously differentiable and $\nabla f$ is locally Lipschitz continuous, respectively; while in some relevant instances of~\eqref{eqn:prob} B- and P-stationarity coincide, M-stationarity is often verified at infinitely many points that are not B-stationary (see the discussion in~\cite[Sec.\ 7]{olikier2025projected}).

Under global $L$-smoothness assumptions, quite recent studies~\cite{themelis2018forward,pauwels2024note} proved that the PGM with constant stepsize converges to P-stationary points of problem~\eqref{eqn:prob}, substantially generalizing the result known for sparsity-constrained optimization~\cite{beck2013sparsity}.
As pointed out in~\cite{de2022proximal}, however, the global Lipschitz smoothness assumption is impractical for some of the major applications of the PGM, such as the use as an internal solver in augmented Lagrangian or sequential penalty type frameworks~\cite{jia2023augmented,kanzow2023inexact}.
For this reason, the algorithm, equipped with a nonmonotone line search, was studied in the absence of any smoothness assumption in~\cite{jia2023augmented}, where convergence to M-stationarity is guaranteed. 
This result, however, was not enough to rule out the occurrence of \textit{apocalypses}~\cite{levin2023finding}, i.e., situations where approximate stationarity of the solutions, measured according to B-stationarity definition, tends to zero, but the limit point is only M-stationary. 
Then, the analysis carried out in~\cite{olikier2025projected} allowed to fill the gap, proving that PGM with nonmonotone line searches converges to B-stationary points in general, and to P-stationary points under the mild local Lipschitz smoothness assumption.
A question left open, however, concerns the properties of the method when an arbitrary descent (or at least gradient-related) direction is employed as a reference direction in the update rule, in place of the negative gradient. 

The employment of the FBE~\cite{patrinos2013proximal} allows one to globalize proximal gradient variants that make use of directions associated with fast local convergence~\cite{themelis2018forward,bonettini2020convergence}; the PANOC algorithm~\cite{stella2017simple}, in particular, is guaranteed to converge to P-stationary points under global $L$-smoothness hypotheses, whereas the PANOC+ method described in~\cite{de2022proximal} can be shown to converge to M-stationary points under local smoothness. These properties are achieved  by a double ``quality check'', each one associated with the progressive reduction of a parameter: one condition to be attained accounts for the goodness of the stepsize, while the other monitors the quality of the search direction. The two backtracking procedures are intertwined with each other, and this mechanism might in the end be not ideal from the computational side.

On the other hand, we will point out in this work that the straightforward extension of the PGM suggested at the end of~\cite{olikier2025projected} is not actually sound. 
Throughout the manuscript, we then investigate the enhancement of the spectral projected gradient method of~\cite{jia2023augmented} modifying the base search direction with the addition of a momentum term~\cite{polyak1964some,lapucci2026globally}, with the goal of proving convergence results without global Lipschitz assumptions and without resorting to the FBE framework. 
The addition of momentum terms in iterative update rules is well-known to be beneficial in unconstrained optimization, especially in the (strongly) convex case~\cite{ghadimi2015global} and when working with incremental approaches for finite-sum problems related to data science and machine learning tasks~\cite{gitman2019understanding,sutskever2013importance}. In the nonconvex case, the closely related class of conjugate gradient methods had vast success~\cite[Ch.\ 12]{grippo2023introduction}, whereas the definition of globally convergent methods based on heavy-ball type terms is more recent~\cite{lee2022limited,lapucci2026globally}. Similar strategies were then furthermore adapted for the context of Riemannian optimization~\cite{tang2024riemannian,leggio2026riemannian} and within projected gradient approaches for general convex constraints~\cite{lapucci2026projected}.

The resulting framework, which relies on a curve backtracking mechanism~\cite{donnini2026efficient,jia2026projection,anonymous2026safely} for the pre-projection step, can be proved to converge: i) to M-stationary points of problem~\eqref{eqn:prob} without any smoothness assumptions, as long as the base directions asymptotically revert to the negative gradient for small stepsizes and ii) to B-stationary points without smoothness and P-stationary points under local smoothness if the base search direction reverts exactly to the negative gradient for very small positive steps.

We shall notice that the nonconvexity of the feasible set $D$ makes the straight generalization of the strategy in~\cite{lapucci2026projected} impossible, not just because an arbitrary convex combination of the projected gradient and projected momentum directions can likely result in an infeasible search direction, but most importantly because the projection arc induced by such direction could generate an ascent path. As we will detail, the latter is in fact the reason why PGM cannot be plainly employed with a base (descent) direction different than the negative gradient.

The rest of the paper is organized as follows: in Section~\ref{sec:prelim} we summarize some well-established concepts from constrained optimization and variational analysis that are needed for the core discussion of the paper; then, in Section~\ref{sec:example}, we provide an example to highlight the shortcomings possibly occurring when substituting the negative gradient direction in PGM with a generic gradient-related direction; in Section  \ref{sec:alg} we describe the proposed algorithmic framework, proving convergence to M-, B-, and P-stationary points in Section~\ref{sec:Mstat}, \ref{sec:Bstat} and \ref{sec:Pstat} respectively. We then show the results of computational experiments in Section~\ref{sec:experiments} and finally give concluding remarks in Section~\ref{sec:conclusions}.

\section{Preliminaries}
\label{sec:prelim}
Euclidean projections onto a closed, nonempty but nonconvex set $D\subset \mathbb{X}$ exist, although, differently from the convex case, they are not necessarily unique. The set-valued projection operator onto $D$ is therefore defined as $P_D:\mathbb{X}\rightrightarrows D$, where
$$P_D(x) = \argmin_{z\in D} \|z-x\|.$$
For all $x$, the set $P_D(x)$ is nonempty and compact; moreover, the corresponding distance (single-valued) function $d_D:\mathbb{X}\to\mathbb{R}_+$, where $d_D(x) = \min_{z\in D}\|z-x\|$, is continuous.  For all $x \in D$, $v \in \mathbb X$, and $y \in P_D(x-v)$, it is verified that (see, e.g.,~\cite[Prop.~2.1]{olikier2025projected})
\begin{align}
\label{eqn:propr_proj1}
\|y-x\| &\le 2 \|v\|; \\
\label{eqn:propr_proj2}
2\langle v, y-x \rangle &\le -\|y-x\|^2.
\end{align}

We now review the three stationarity notions for problem~\eqref{eqn:prob} considered in the paper. To this aim, we first need to recall some basic concepts from variational analysis. We refer the reader to~\cite[Ch.\ 6]{rockafellar1998variational},~\cite[Sec.\ 2]{olikier2025projected}, or~\cite[Ch.\ 1]{mordukhovich2018variational} for more details on the following concepts. We start with the notion of tangency.

\medskip
\begin{definition}
    A vector $v \in \mathbb{X}$ is  \textit{tangent} to $D$ at a point $x \in D$ if there exist sequences $\{x_k\}\subseteq D$ and $\{t_k\}\subseteq \mathbb{R}_+$ such that $$\lim_{k\to\infty}x_k = x, \qquad \lim_{k\to\infty}t_k=0,\qquad \text{and}\qquad \lim_{k\to\infty}\frac{x_k-x}{t_k} = v.$$ 
\end{definition}

The above definition allows us to introduce the following objects, which can be proved to be a closed cone and a closed convex cone, respectively.

\medskip
\begin{definition}
    Let $x\in D$ be any feasible point.
    \begin{itemize}
    \item The \textit{tangent cone} to $D$ at $x$, denoted by $T_{D}(x)$, is the set of all tangent vectors to $D$ at $x$.
    \item The \textit{regular (Fréchet) normal cone} to $D$ at $x$, denoted by $\widehat{\mathcal{N}}_D(x)$, is the set of vectors that form a nonacute angle with every tangent vector at the point, i.e., 
    $$\widehat{\mathcal{N}}_D(x)=\{v\in\mathbb{X}\mid \langle v,w\rangle\le 0\;\forall w\in T_D(x)\}.$$
    \end{itemize}
\end{definition}

Then, once the normal cone $\widehat{\mathcal{N}}_D(x)$ is well-defined, normality of vectors can also be characterized more mildly, in the limit.

\medskip
\begin{definition}
    A vector $v \in \mathbb{X}$ is  \textit{normal (in the general sense)} to $D$ at $x \in D$
    if there exist sequences $\{x_k\}\subseteq D$ and $\{v_k\}\subseteq \mathbb{X}$  such that $$\lim_{k\to\infty}x_k = x, \qquad \lim_{k\to\infty}v_k=v,\qquad \text{and}\qquad v_k\in\widehat{\mathcal{N}}_D(x_k)\;\forall\, k.$$ 
\end{definition}

From the above notion of normality, we immediately get a set which can be proved to be a closed cone and which somewhat represents a ``relaxation'' of the regular normal cone. 

\medskip
\begin{definition}
    Let $x\in D$ be a feasible point. 
    The \textit{general (Mordukhovich/limiting) normal cone} to $D$ at $x$, denoted by $\mathcal{N}_D(x)$, is the set of normal vectors to $D$ at $x$.
\end{definition}

\medskip
We finally have a series of concepts that are made accessible by the availability of the projection operator.

\medskip
\begin{definition}
\label{def:prox_normal}
A vector $v \in \mathbb{X}$ is \emph{proximal normal} to $D$ at $x \in D$ if there exists $\bar{\alpha}>0$ such that $x \in P_{D}(x+\bar{\alpha}v)$.    
\end{definition}

\medskip
\begin{remark}
\label{obs:equiv_prox}
The latter Definition~\ref{def:prox_normal} can be shown to 
\begin{enumerate}
\item[i)] be equivalent to $\bar{\alpha}\|v\|=d_D(x+\bar{\alpha} v)$; and \item[ii)] actually implies that $P_{D}(x+\alpha v) = \{x\}$ for all $\alpha\in (0,\bar{\alpha})$. 
\end{enumerate}
\end{remark}

\medskip
This leads to the definition of a final set that can be proved to be a closed convex cone.

\medskip
\begin{definition}
    Let $x\in D$ be any feasible point. 
    The \textit{proximal normal cone} to $D$ at $x$, denoted by $\widehat{\mathcal{N}}_D^P(x)$, is the set of proximal normal vectors to $D$ at $x$.
\end{definition}

\medskip
We can now relate formally the three notions of normal cone with each other.

\medskip
\begin{proposition}
    Let $x\in D$ be any feasible point. The following inclusions hold:
    \begin{equation}
        \label{eqn:inclusions}
        \widehat{\mathcal{N}}_D^P(x)\subseteq \widehat{\mathcal{N}}_D(x)\subseteq\mathcal{N}_D(x).
    \end{equation}
\end{proposition}
As aforementioned in the Introduction, the three normal cones give rise to three different concepts of stationarity (i.e., first order necessary conditions of local optimality) for problem~\eqref{eqn:prob}, so that a point $x\in D$ is

\smallskip
\begin{itemize}
    \item \textit{M-stationary} if $-\nabla f(x)\in \mathcal{N}_D(x)$;
    \item \textit{B-stationary} if $-\nabla f(x)\in \widehat{\mathcal{N}}_D(x)$;
    \item \textit{P-stationary} if $-\nabla f(x)\in \widehat{\mathcal{N}}^P_D(x)$.
\end{itemize}

\smallskip
By the inclusions in~\eqref{eqn:inclusions}, it is clear that P-stationarity implies B-stationarity which in turn implies M-stationarity. M-stationarity is equivalent to B-stationarity at a point $x \in D$ if and only if $D$ is Clarke regular at $x$~\cite[Def.\ 6.4]{rockafellar1998variational}, which is false in various relevant instances, as analyzed in~\cite[Sec.\ 7]{olikier2025projected}, making M-stationarity strictly weaker than B-stationarity in practice. By the same discussion, we get that the coincidence of B- and P-stationarity is instead satisfied in settings of interest; still an example where both inclusions in~\eqref{eqn:inclusions} are strict is provided.

\section{Projection along a Descent Direction can Induce Ascent Paths}
\label{sec:example}
The work of Olikier and Waldspurger on the convergence properties of PGM algorithms raises an additional research question~\cite{olikier2025projected}, which is rephrased here below for notation consistency with that of this manuscript:

\smallskip
\begin{center}
    \textit{``Can the convergence results proved for PGM be extended to an algorithm that uses other descent directions than the negative gradient? For example, a search direction at a point $x \in D$ that is not B-stationary could be a vector $v \notin \widehat{\mathcal{N}}_{D}^P(x)$ that satisfies $\langle\nabla f(x),v\rangle \le - c_1 \|\nabla f(x)\|^2$ and $\|{v}\| \le c_2 \|\nabla f(x)\|$ with $c_1, c_2>0$.''}
\end{center}

\medskip
In this section we show by a simple counterexample that the answer, at least for a vanilla extension of PGM using a direction as suggested above in place of the negative gradient, is negative: a method carrying out iterations of the form $x_{k+1} = P_D(x_k+t_kd_k)$ with $d_k$ gradient-related and $t_k$ identified by a backtracking procedure ensuring (possibly nonmonotone) sufficient decrease can get stuck at points that are not even M-stationary solutions. 
The intuition about this issue, depicted in Figure~\ref{fig:counterexample}, is formalized in the following example.

\begin{figure}
    \centering
    \includegraphics[width=0.75\linewidth]{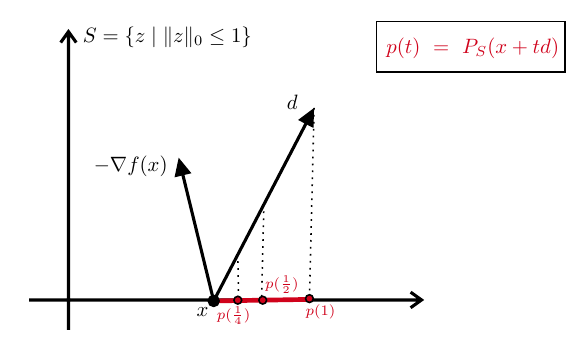}
    \caption{Direction $d$ is a descent direction at $x$ (forms an acute angle with negative gradient), but projected gradient path $p(t)$ generates an ascent direction, so that for sufficiently small $t$ the objective increases.}
    \label{fig:counterexample}
\end{figure}

\medskip
\begin{example}
\label{ex:sparsity}
    Let us consider the problem 
    \begin{equation}
        \label{eqn:counterexample}
        \min_{x\in\mathbb{R}^2}\; f(x) = \frac{x_1^2}{2}+x_2^2-3x_2\qquad \text{s.t. } \|x\|_0\le 1.
    \end{equation}
    For the class of sparsity-constrained optimization problems in $\mathbb{R}^n$, where $D = \{x\mid \|x\|_0\le s\}$, we have (see for references~\cite[Sec.\ 5.1]{kanzow2023inexact},~\cite[Sec.\ 3.3]{lammel2022nondegenerate},~\cite[Sec.\ 7.4]{olikier2023first}):

    \smallskip
    \begin{itemize}
        \item a point $\bar{x}$ is M-stationary if and only if there exists a subset of indices $J\subseteq\{1,\ldots,n\}$ such that 
        $$\begin{cases}
            |J| = s&\text{(} J \text{ identifies a set of } s \text{ variables)}\\
            i\in J\qquad\quad\;\;\, \forall \, i:\bar{x}_i\neq 0&\text{(}J\text{ contains all active variables at }\bar{x}\text{)}\\
            \nabla_i f(\bar{x}) = 0 \quad\forall\, i\in J&\text{(zero gradients w.r.t.\  variables in } J\text{)}
        \end{cases}$$ 
        (i.e., $J$ identifies a set of $s$ variables containing the support of $\bar{x}$) and
        \item B- and P-stationarity coincide and $\bar{x}\in D$ is B-stationary if and only if $$\begin{cases}
            \nabla f(\bar{x}) = 0 &\text{if } \|\bar{x}\|_0<s;\\
            \nabla_i f(\bar{x}) = 0 \;\forall\, i:\bar{x}_i\neq 0&\text{if }\|\bar{x}\|_0=s.
        \end{cases}$$
    \end{itemize}
    It is immediate to observe that M- and B-stationarity coincide at all points with full support, i.e., whenever $\|x\|_0=s$.
    
    Focusing back on problem~\eqref{eqn:counterexample}, let us consider point $\bar{x} = (1,0)^\top$. We have $\nabla f(x) = (x_1, 2x_2-3)^\top$, so the gradient at $\bar{x}$ is  $\nabla f(\bar{x}) = (1,-3)^\top$. The point $\bar{x}$ is clearly not stationary, as the derivative w.r.t.\ the active variable $x_1$ is nonzero. Now, let us consider direction $d=(1,3)^\top.$ This is a descent direction at $\bar{x}$, as $\nabla f(\bar{x})^\top d = -8$. 
    On the other hand, we have
    $$P_D(\bar{x}+td) = P_D ((1+t,3t)^\top)  = (1+t,0)^\top \qquad \forall\,t\in(0,1/2),$$
          and $$f((1+t,0)^\top) = \frac{(t+1)^2}{2}>\frac{1}{2}=f(\bar{x})\qquad \forall\, t>0.$$
          The projected path induced by the descent direction $d$ at $\bar{x}$ is thus an ascent one, and the algorithm would get stuck at the non-stationary point $\bar{x}$, with the backtracking procedure possibly trapped in an infinite loop.  
\end{example}

\medskip
The above example leads us to the need of devising a more structured approach in order to enable the inclusion of momentum terms within PGM frameworks for problem~\eqref{eqn:prob}.

\section{A Projection-based Method with Momentum}
\label{sec:alg}
As widely anticipated, the main goal of this paper consists in the design of a sound projected-type algorithm exploiting momentum terms for computational acceleration. In other words, we are interested in an algorithmic update of the form 
\begin{equation*}
    x_{k+1} \in P_D(x_k-\alpha_k \nabla f(x_k)+\beta_k(x_k-x_{k-1})), 
\end{equation*}
defined in such a way to avoid stall situations like the one described in the previous section.

To this aim, we devise a scheme that moves the tentative point to be projected by backtracking along a tailored curve, rather than a straight direction, until a suitable sufficient decrease condition is verified. 
The proposed update rule, specifically, is defined as
\begin{equation*}
    x_{k+1} = P_D(x_k+d_k(\alpha_k))
\end{equation*}
with 
\begin{equation*}
    d_k(\alpha_k) = -\alpha_k\nabla f(x_k)+\beta_\text{base}\big(\tfrac{\max\{\alpha_k-\tau,0\}}{\alpha_\text{base}-\tau}\big)^{a}(x_k-x_{k-1})
\end{equation*}

\smallskip\noindent
for values $a\ge 1$, $\tau\ge 0$, $\alpha_\text{base} > \tau$ and $\beta_\text{base}\ge0$. 
The value of $\alpha_k$ will then  be obtained by a backtracking procedure, started at $\alpha=\alpha_\text{base}$, that aims to satisfy the following sufficient decrease condition~\cite{kanzow2022convergence,de2023proximal}
\begin{equation}
    \label{eqn:suff_dec_cond}
    f(P_D(x_k+d_k(\alpha)))\le \mu_k-\frac{c}{2\alpha}\|P_D(x_k+d_k(\alpha))-x_k\|^2,
\end{equation}
where $c>0$ and $\mu_k\ge f(x_k)$ is a suitable (possibly nonmonotone) reference value.
The two nonmonotone decrease conditions that we will specifically consider in this work are:

\smallskip
\begin{enumerate}
    \item The \emph{max-rule} for some  $M \in \mathbb N$~\cite{grippo1986nonmonotone}:
\begin{equation}
\label{eqn:max_rule}
    \mu_k = \max\limits_{i \in \{\max\{0,k-M\},\dots,k\}}f(x_i);
\end{equation}
    \item The \emph{average-rule} for some $ p \in (0,1]$~\cite{zhang2004nonmonotone}:
\begin{equation}
\label{eqn:average_rule}
    \mu_k = (1-p)\mu_{k-1} + p f(x_k).
\end{equation}
\end{enumerate}

The logic associated with the proposed update mechanism, illustrated in Figures~\ref{fig:curve1} and \ref{fig:curve2}, then works as follows:

\smallskip
\begin{itemize}
    \item The search for the update starts at $\alpha=\alpha_\text{base}$, i.e., the first trial point corresponds to the update $P_D(x_k-\alpha_\text{base}\nabla f(x_k)+\beta_\text{base} (x_k-x_{k-1}))$ that we ideally would like to take.
    \item The sufficient decrease condition~\eqref{eqn:suff_dec_cond} is checked at each tentative point; backtracking steps are performed as long as the condition is not met. 
    \item As backtracking proceeds, we progressively realign with the path induced by using the negative gradient direction; in particular, we have two cases:

    \smallskip
    \begin{enumerate}[a)]
        \item if $\tau = 0$, we need $a\ge 2$; in this case, the momentum term vanishes faster than the negative gradient in $d_k(\alpha)$ as $\alpha$ tends to zero, and tentative updates revert to the ones of PGM asymptotically (Figure~\ref{fig:curve1}); we shall note that this type of approach borrows the idea discussed in~\cite{donnini2026efficient,anonymous2026safely}, and is also similar to the mechanism employed within PANOC~\cite{stella2017simple,de2022proximal} to handle general directions;
        \item if $\tau>0$, we exactly reduce to PGM steps for sufficiently small step sizes (Figure~\ref{fig:curve2}); in this case, any value of $a\ge 1$ is acceptable.
    \end{enumerate}
\end{itemize}

\smallskip
The proposed algorithmic framework is formalized in detail in Algorithm~\ref{alg:ggmm}, where the backtracking procedure is stated explicitly as an inner loop. It is worth remarking that the sufficient decrease condition~\eqref{eqn:suff_dec_cond}, which we will later prove to be satisfiable with a finite number of backtracks in the setups considered in this manuscript, is stricter than the conditions used both in~\cite{jia2023augmented} and~\cite{olikier2025projected}. However, it enforces the property that is actually crucial to be used in the convergence analysis and that is retrieved also in the two aforementioned works.

For a more general and flexible framework, we allow in Algorithm~\ref{alg:ggmm} the values of $\alpha_\text{base}$ and $\beta_\text{base}$ to change across iterations considered (we will denote these starting values by $\alpha_{k0}$ and  $\beta_{k0}$). Some options, like the target stationarity conditions or the sufficient decrease safeguard, are left flexible for generality; we will then specialize the convergence analysis for specific choices of these options and also depending on the setup of algorithmic parameters.

In summary, in the next sections we are going to prove that

\smallskip
\begin{itemize}
    \item for the case $\tau=0$, where the update never exactly collapses to PGM, we recover convergence to M-stationarity with both max-type and average-type nonmonotone safeguard rules without any smoothness assumption on $\nabla f$; 
    we shall observe that, for instance, convergence for PANOC+ in this scenario has not been proven;
    \item if we set $\tau>0$, we recover the strong convergence properties of PGM, namely, convergence to B-stationary points without smoothness assumptions and convergence to P-stationary points under local smoothness of $\nabla f$; the result is again proved for both types of nonmonotone conditions.
\end{itemize}

\begin{algorithm2e}[htb]
    \SetKwInOut{Input}{input}
	\SetAlgoLined
    \Input{$x_0 \in D, 0 \le \tau < \alpha_{\min}\le\alpha_{\max}<\infty, \beta_{\max}\ge 0, c \in (0,1), \delta \in (0,1), a \ge 1$}
    $k \gets 0$\;
    $x_{-1} \gets x_0$\;
	\While{\label{line:while_start} $x_k$ is not (M/B/P)-stationary}{
        Choose $\alpha_{k0} \in [\alpha_{\min},\alpha_{\max}]$ and $\beta_{k0}\in [0,\beta_{\max}]$\; \label{line:initialparameters}
        Set $\mu_k\ge f(x_k)$\;  \label{line:muk}
        Choose $y_{k0}\in P_D(x_k-\alpha_{k0} \nabla f(x_k) + \beta_{k0} (x_k-x_{k-1}))$\;
        $j \gets 0$\;
		\While
		{
			\label{line:linesearch_start}%
			$
			f(y_{kj}) > \mu_k - \frac{c}{2\alpha_{kj}}\|y_{kj}-x_k\|^2$
		}
		{
        $j \gets j+1$\;
        $\alpha_{kj} \gets  \delta^j \alpha_{k0}$\;
        $\beta_{kj} \gets \beta_{k0}\big(\frac{\max\{\alpha_{kj}-\tau,0\}}{\alpha_{k0}-\tau}\big)^{a}$\;
        Choose $y_{kj}\in P_D(x_k-\alpha_{kj} \nabla f(x_k) + \beta_{kj} (x_k-x_{k-1}))$\;
		}
        \label{line:linesearch_end}
        $\alpha_k \gets \alpha_{kj}$\;
        $\beta_k \gets \beta_{kj}$\;
		$x_{k+1} \gets y_{kj}$\;
        $k \gets k+1$\;
	}
    \label{line:while_end}
	\caption{\texttt{Geometric Gradient Method with Momentum}}
	\label{alg:ggmm}
\end{algorithm2e}

\begin{figure}[htbp]
    \centering
    \includegraphics[width=0.9\linewidth]{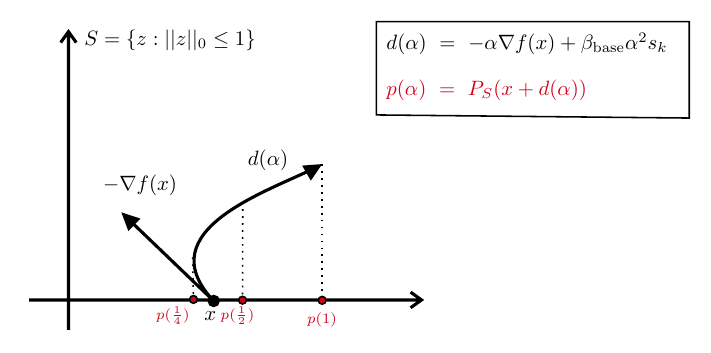}
    \caption{Search path for backtracking induced by the proposed update rule when $\tau=0$, $a=2$ and $\alpha_\text{base} = 1$.}
    \label{fig:curve1}
\end{figure}

\begin{figure}[htbp]
    \centering
    \includegraphics[width=0.9\linewidth]{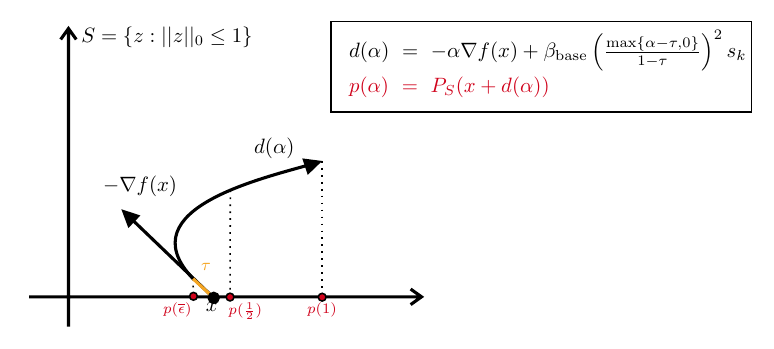}
    \caption{Search path for backtracking induced by the proposed update rule when $\tau>0$, $a=2$ and $\alpha_\text{base} = 1$.}
    \label{fig:curve2}
\end{figure}

\subsection{Convergence to M-stationary points for $\tau=0$}
\label{sec:Mstat}
In this section, we will carry out the analysis of Algorithm~\ref{alg:ggmm} under the assumption that $\tau$ is set to 0 and that $a\ge2$; we will also assume that, in the condition at Line~\ref{line:while_start}, M-stationarity is considered. 

For every iteration $k$ and $j$ of the outer and inner while loop, respectively, the parameter $\beta_{kj}$ can be written as
$$\beta_{kj} = \beta_{k0}\big(\frac{\alpha_{kj}}{\alpha_{k0}}\big)^{a} = \beta_{k0}\delta^{aj},$$
which, for a given $k$ and for $j \to \infty$, tends to zero \textit{faster} than $\alpha_{kj}=\alpha_{k0}\delta^j$. Indeed, as $a\ge2$,
\begin{equation}
    \label{eqn:lim_ratio}
    \lim_{j \to \infty}\frac{\beta_{kj}}{\alpha_{kj}}=\lim_{j \to \infty} \frac{\beta_{k0}}{\alpha_{k0}}\delta^{j(a-1)} = 0.
\end{equation}
Let us now consider the quadratic function $Q_{kj}:\mathbb X \to \mathbb R$ defined as

\begin{equation}
\label{eqn:Qkj}
Q_{kj}(y) :=\langle \nabla f(x_k) - \frac{\beta_{kj}}{\alpha_{kj}}(x_k-x_{k-1}), y-x_k \rangle + \frac{1}{2\alpha_{kj}}\| y - x_k\|^2,
\end{equation}
where $x_k\in D$, and $\alpha_{kj}, \beta_{kj}\in \mathbb R$ are computed according to the instructions of the algorithm. Then, 
\begin{equation}
\label{eqn:Q_cost}
Q_{kj}(x_k)=0.
\end{equation}
Moreover, it is straightforward to show that selecting
$$y_{kj} \in P_D(x_k - \alpha_{kj} \nabla f(x_k) + \beta_{kj} (x_k-x_{k-1}))$$
is completely equivalent to choosing
$$y_{kj} \in \arg\min_{y \in D} Q_{kj}(y).$$
It follows that, for every $k$ and $j$, $y_{kj}$ is an M-stationary point for the problem
$\min_{y \in D}Q_{kj}(y)$,
i.e.,
\begin{equation}
\label{eqn:M-subprob}
-\nabla f(x_k)-\frac{1}{\alpha_{kj}}(y_{kj}-x_k)+\frac{\beta_{kj}}{\alpha_{kj}}(x_k-x_{k-1}) \in \mathcal N_D(y_{kj}).
\end{equation}

We are now ready to state a first result, concerning the correctness of the backtracking procedure defined by the  inner loop at Lines~\ref{line:linesearch_start}--\ref{line:linesearch_end} of Algorithm~\ref{alg:ggmm}.

\medskip
\begin{proposition}
\label{prop:armijo}
Let $a\ge2$ and $\tau=0$. Assume that $f$ is continuously differentiable. Fix an outer iteration $k$ of Algorithm~\ref{alg:ggmm} and suppose that $-\nabla f(x_k) \not \in \mathcal N_D(x_k)$, i.e., $x_k$ is not an M-stationary point of problem~\eqref{eqn:prob}. Then, there exists $\bar j\in \mathbb N$ such that
$$f(y_{k \bar j}) \le \mu_k - \frac{c}{2\alpha_{k \bar j}}\|y_{k \bar j}-x_k\|^2.$$
\end{proposition}

\begin{proof}
Suppose, by contradiction, that the loop at Lines~\ref{line:linesearch_start}--\ref{line:linesearch_end} does not terminate, i.e., for all $j \in \mathbb N$,
\begin{equation}
\label{eqn:assurdo_}
f(y_{kj}) > \mu_k - \frac{c}{2\alpha_{kj}}\|y_{kj}-x_k\|^2.
\end{equation}
Recall that $y_{kj} \in \arg\min_{y\in D}Q_{kj}(y)$.	Hence, as $ x_k \in D $, $Q_{kj}(y_{kj}) \le Q_{kj}(x_k)$, i.e., recalling~\eqref{eqn:Qkj}--\eqref{eqn:Q_cost},
\begin{equation}
\label{eqn:dis_quad}
\langle \nabla f(x_k) - \frac{\beta_{kj}}{\alpha_{kj}}(x_k-x_{k-1}), y_{kj}-x_k \rangle + \frac{1}{2\alpha_{kj}}\| y_{kj} - x_k\|^2 \le 0.
\end{equation}
Rearranging, applying the Cauchy--Schwarz and the triangle inequalities and dividing by $\|y_{kj}-x_k\|$, we therefore get 
\begin{align*}	
\frac{1}{2\alpha_{kj}}\| y_{kj} - x_k\| & \le \|  \nabla f(x_k) - \frac{\beta_{kj}}{\alpha_{kj}}(x_k-x_{k-1}) \|  \le \|  \nabla f(x_k) \| + \frac{\beta_{k0}}{\alpha_{k0}} \delta^{j(a-1)} \|x_k - x_{k-1}\| \\
&\le  \|  \nabla f(x_k) \| + \frac{\beta_{\max}}{\alpha_{\min}} \|x_k - x_{k-1}\|.
\end{align*}
Taking the limit for $j\to\infty$, noting that $\lim_{j \to \infty}{\alpha_{kj}} = 0$ and that the rightmost term in the above chain of inequalities is a constant, we get that $\lim_{j \to \infty} y_{kj} = x_k$.
Now, we distinguish two cases:

\smallskip
\begin{enumerate}
\item
The limit $\lim_{j \to \infty} \frac{\|y_{kj}-x_k\|}{\alpha_{kj}} = 0$ holds. Then, taking the limits in~\eqref{eqn:M-subprob} and recalling~\eqref{eqn:lim_ratio}, we immediately obtain, by the outer semicontinuity of the general normal cone, that
$$-\nabla f(x_k)\in \mathcal N_D(x_k).$$
This contradicts the hypothesis that $-\nabla f(x_k)\not\in \mathcal N_D(x_k)$.

\item We have $\limsup_{j \to \infty} \frac{\|y_{kj}-x_k\|}{\alpha_{kj}} > 0$. Then, there exists a constant $\rho>0$ and an infinite subset $J \subset \mathbb N$ such that 
\begin{equation}
    \label{eqn:rewrite_limsup}
    \frac{\|y_{kj}-x_k\|}{\alpha_{kj}} \ge \rho \quad \forall\,j \in J.
\end{equation}
From the Taylor expansion,
\begin{equation}
\label{eqn:taylor}
f(y_{kj}) - f(x_k)  = \langle \nabla f(x_k), y_{kj}-x_k \rangle  + o(\|y_{kj}-x_k\|).
\end{equation}
We can now observe that there exists $j_1 \in\mathbb N$ such that for all $j \ge j_1$, 
$$o(\|y_{kj}-x_k\|) \le  \frac{(1-c)\rho}{4}\|y_{kj}-x_k\|.$$
Then, also using~\eqref{eqn:rewrite_limsup}, for all $j \in J$, $j \ge j_1$, 
\begin{equation}
\label{eqn:bound1}
o(\|y_{kj}-x_k\|) \le \frac{1-c}{4\alpha_{kj}}\|y_{kj}-x_k\|^2.
\end{equation}
Moreover, $j_2 \in \mathbb N$ exists such that for all $j \ge j_2$, 
$$\frac{\beta_{kj}}{\alpha_{kj}} = \frac{\beta_{k0}}{\alpha_{k0}}\delta^{j(a-1)} \le \frac{(1-c)\rho}{4\|x_k-x_{k-1}\|}.$$
Then, exploiting the above inequality together with Cauchy--Schwarz and~\eqref{eqn:rewrite_limsup}, we get for all $j \in J, j \ge j_2$,
\begin{align*}
\frac{\beta_{kj}}{\alpha_{kj}}\langle x_k-x_{k-1}, y_{kj}-x_k \rangle &\le  \frac{\beta_{kj}}{\alpha_{kj}} \|x_k-x_{k-1}\| \|y_{kj}-x_k\| \\
& \le \frac{(1-c)\rho}{4}\|y_{kj}-x_k\| \le \frac{1-c}{4\alpha_{kj}}\|y_{kj}-x_k\|^2.
\end{align*}
Recalling~\eqref{eqn:dis_quad}, we thus get
\begin{equation}
\label{eqn:bound2}
\begin{aligned}
 \langle \nabla f(x_k), y_{kj}-x_k \rangle  &\le \frac{\beta_{kj}}{\alpha_{kj}}\langle x_k-x_{k-1}, y_{kj}-x_k \rangle -\frac{1}{2\alpha_{kj}}\| y_{kj} - x_k\|^2 \\
&\le  \frac{1-c}{4\alpha_{kj}}\|y_{kj}-x_k\|^2 -\frac{1}{2\alpha_{kj}}\| y_{kj} - x_k\|^2 \\&= -\frac{1+c}{4\alpha_{kj}}\|y_{kj}-x_k\|^2.
\end{aligned}
\end{equation}
Combining~\eqref{eqn:taylor},~\eqref{eqn:bound1} and~\eqref{eqn:bound2}, and recalling that $\mu_k\ge f(x_k)$, for all $j \in J$, $j \ge \max\{j_1,j_2\}$ we have
\begin{align*}
  f(y_{kj}) - \mu_k&\le f(y_{kj}) - f(x_k) \\&\le  -\frac{1+c}{4\alpha_{kj}}\|y_{kj}-x_k\|^2 + \frac{1-c}{4\alpha_{kj}}\|y_{kj}-x_k\|^2 \\&=   -\frac{c}{2\alpha_{kj}}\|y_{kj}-x_k\|^2,
\end{align*}
which contradicts~\eqref{eqn:assurdo_}. \qedhere
\end{enumerate}
\end{proof}

\medskip
We now turn to the study of the properties of the sequence $\{x_k\}$ associated with the outer algorithmic loop.
The first property, stated as Proposition~\ref{prop:conv_diff}, is valid for every $a \ge 1$ and $\tau \ge 0$; hence, it  will be useful also for the analysis presented in the next subsection. This result, whose proof is based on those of~\cite[Prop.~3.2]{jia2023augmented} and~\cite[Prop.~5.4]{olikier2025projected} and provided in Appendix \ref{sec:appendix} for completeness, ensures that the difference between two consecutive iterates tends to zero if either the max or the average nonmonotone decrease rule is considered.
Specifically, in the following we will require that:

\medskip
\begin{assumption}
\label{ass:rule}
One of the following is true:
\smallskip
\begin{enumerate}
\item The \emph{max-rule}~\eqref{eqn:max_rule} is used in Line~\ref{line:muk} of Algorithm~\ref{alg:ggmm}; $f$ is bounded from below and uniformly continuous on the sublevel set 
$$\mathcal L_0 = \{x \in D \mid f(x) \le f(x_0)\}.$$ 
\item The \emph{average-rule}~\eqref{eqn:average_rule} is used in Line~\ref{line:muk} of Algorithm~\ref{alg:ggmm}.
\end{enumerate}
\end{assumption}

\medskip
\begin{proposition}
\label{prop:conv_diff}
Assume that $f$ is continuously differentiable, and that every inner loop (Lines~\ref{line:linesearch_start}--\ref{line:linesearch_end}) of Algorithm~\ref{alg:ggmm} always terminates, while the outer loop (Lines~\ref{line:while_start}--\ref{line:while_end}) never terminates.  Let $\{x_k\}$ be the corresponding infinite sequence of iterates. Assume that the sequence $\{x_k\}$ admits an accumulation point. Further assume Assumption~\ref{ass:rule}.
Then, 
\begin{equation}
\label{eqn:conv_diff}
\lim_{k \to \infty}\|x_{k+1}-x_k\|=0.
\end{equation}
\end{proposition}
\begin{proof}
    See Appendix \ref{sec:appendix}.
\end{proof}

\medskip
Now, we can further characterize the behavior of $\|x_{k+1}-x_k\|$, putting it in relationship with the behavior of $\alpha_k$, along convergent subsequences.

\medskip
\begin{proposition}
\label{prop:conv_rapporto}
Let $a\ge2,\tau=0$. Assume that $f$ is continuously differentiable, and that the outer loop (Lines~\ref{line:while_start}--\ref{line:while_end}) of Algorithm~\ref{alg:ggmm} never terminates. Further assume Assumption~\ref{ass:rule}.  Let $\bar x \in D$ be an accumulation point of the sequence $\{x_k\}$ (if any), i.e., $\{x_k\}_{k \in K}$ converges to $\bar x$ along a subsequence $K \subset \mathbb N$. Then, 
\begin{equation}
\label{eqn:tesi_conv_rapporto}
\lim_{k \in K, k \to \infty}\frac{\|x_{k+1}-x_k\|}{\alpha_k}=0.
\end{equation}
\end{proposition}

\medskip
\begin{proof}
Let $ \bar x \in D $ be an arbitrary accumulation point of $\{x_k\}$. Let $K \subset \mathbb N$ be an infinite subset such that the sequence $\{x_k\}_{k \in K}$ converges to $\bar x$. 
Suppose, by contradiction, that~\eqref{eqn:tesi_conv_rapporto} does not hold. Then, there exists $\rho >0$ and an infinite subset $K' \subset K$ such that
\begin{equation}
\label{eqn:assurdo}
\frac{\|x_{k+1}-x_k\|}{\alpha_k} \ge \rho \qquad \text{for all } k \in K'.
\end{equation}
First, observe that the sequence $ \{ \alpha_k \}_{k \in K'} $ cannot be bounded away from zero. Indeed, if there existed $\bar \alpha>0$ such that $ \alpha_k \ge \bar\alpha$ for all $k \in K'$, then, by~\eqref{eqn:assurdo} and Proposition~\ref{prop:conv_diff}, we would obtain that, for $k\in K'$, 
$$0 < \rho \bar \alpha \le \rho \alpha_k \le \|x_{k+1}-x_k\| \xrightarrow{k \to \infty} 0,$$ 
which is absurd.
Then, we can find an infinite subset $K'' \subset K'$ such that $\alpha_k < \alpha_{\min}$ for all $k \in K''$ and $\{\alpha_k\}_{k \in K''}$ converges to $0$. For every $k \in K''$, let $j \in \mathbb N$ be the inner iteration of Algorithm~\ref{alg:ggmm} such that 
$$\alpha_k=\alpha_{kj}=\alpha_{k0}\delta^j \qquad \text{and} \qquad \beta_k=\beta_{kj}=\beta_{k0}\delta^{ja}.$$
Here, $j=j(k)$ depends on the outer iteration $k$, but we omit the dependence on $k$ to simplify the notation.  Since $\alpha_k < \alpha_{\min}$, we have $j>0$.
Therefore, we can define 
$$ \hat \alpha_k := \alpha_k / \delta = \delta^{j - 1} \alpha_{k0} = \alpha_{k,j - 1} \qquad \text{and} \qquad \hat \beta_k := \beta_k / \delta^a = \delta^{(j - 1)a} \beta_{k0} = \beta_{k,j - 1}.$$ 
Then, also $\{\hat \alpha_k\}_{k \in K''}$ and $\{\hat\beta_k\}_{k \in K''}$ converge to $0$. Note that 
\begin{equation}
\label{eqn:beta/alpha}
\frac{\hat \beta_k}{\hat\alpha_k}=\frac{\beta_{k0}}{\alpha_{k0}}\delta^{(a-1)(j-1)} \le \frac{\beta_{\max}}{\alpha_{\min}}.
\end{equation}
Moreover, define $\hat x_{k+1}:=x_{k,j-1}$. As $j>0$, $\hat x_{k+1}$ violates the nonmonotone Armijo-type condition, i.e., for all $k \in K''$,
\begin{equation}\label{eqn:armijo_no}
	f(\hat x_{k+1}) > \mu_k - \frac{c}{2\hat\alpha_k} \|\hat x_{k+1}-x_k\|^2 \ge f(x_k)  - \frac{c}{2\hat\alpha_k} \|\hat x_{k+1}-x_k\|^2.
    \end{equation}
Since $ \hat x_{k+1} \in \arg\min_{y \in D} Q_{k,j-1}(y)$, then $Q_{k,j-1}(\hat x_{k+1})\le Q_{k,j-1}(x_k)$. Hence, we have, recalling~\eqref{eqn:Qkj}--\eqref{eqn:Q_cost},

\begin{equation}
\label{eqn:dis_prod_scalare}
\langle \nabla f(x_k) - \frac{\hat \beta_k}{\hat \alpha_k}(x_k-x_{k-1}), \hat x_{k+1}-x_k \rangle + \frac{1}{2\hat \alpha_k}\| \hat x_{k+1} - x_k\|^2 \le 0.
\end{equation}
Rearranging, applying the Cauchy--Schwarz and the triangle inequalities, dividing by $\|\hat x_{k+1}-x_k\|$, and using~\eqref{eqn:beta/alpha}, we therefore get 
\begin{equation}
\label{eqn:dis_rapporto}
\frac{1}{2\hat\alpha_k} \| \hat x_{k+1} - x_k \|   \le \|  \nabla f(x_k) - \frac{\hat \beta_k}{\hat \alpha_k}(x_k-x_{k-1}) \|  \le  \|  \nabla f(x_k) \| + \frac{\beta_{\max}}{\alpha_{\min}} \|x_k - x_{k-1}\|.
\end{equation}
Now, we take the limit for $k$ going to infinity (with $k \in K''$) in~\eqref{eqn:dis_rapporto}. First, observe that $\{\|\nabla f(x_k)\|\}_{k \in K}$ converges to $\| \nabla f(\bar x)\| \in \mathbb R$ (as $\{x_k\}_{k \in K}$ converges to $\bar x$). 
Hence, recalling~\eqref{eqn:conv_diff}, the RHS of~\eqref{eqn:dis_rapporto} converges to a finite quantity. Hence, recalling that $\{\hat \alpha_k\}_{k \in K''}$ tends to $0$, we have 
\begin{equation}
\label{eqn:lim_diff_xk+1}
\lim_{k \in K'', k \to \infty}\|\hat x_{k+1} - x_k\|= 0 \qquad \text{and} \qquad \lim_{k \in K'', k \to \infty}\|\hat  x_{k+1} - x_{k+1}\|= 0. 
\end{equation}
Hence, we conclude that
 $ \{\hat x_{k+1}\}_{k \in K''} $ converges to $\bar x$.
By the mean value theorem, for each $k\in {K''}$, there exists $\xi_k$ on the segment
from $\hat x_{k+1}$ to $x_k$ such that
\begin{equation}
\label{eqn:mean_value}
f(\hat x_{k+1})-f(x_k)= \langle \nabla f(\xi_k), \hat x_{k+1}-x_k\rangle.
\end{equation}
Since both $ \{\hat x_{k+1}\}_{k \in K''}$ and $\{ x_k \}_{k \in K''}$ converge to $\bar x$, we find that 
\begin{equation}
\label{eqn:lim_diff_grad}
\lim_{k \in K'', k \to \infty}\|\nabla f(\xi_k)-\nabla f(x_k)\| = 0.
\end{equation}
For all $k \in K''$, we get
\begin{align*}
	 & - \frac{c}{2\hat\alpha_k} \|\hat x_{k+1}-x_k\|^2
	 <
	f(\hat x_{k+1}) - f(x_k) = \langle \nabla f(\xi_k), \hat x_{k+1}-x_k\rangle \\
     &=\langle \nabla f(x_k), \hat x_{k+1} - x_k \rangle
	+
	\langle \nabla f(\xi_k) - \nabla f(x_k) , \hat x_{k+1} - x_k\rangle \\
	&\le
	\frac{\hat \beta_k}{\hat \alpha_k}\langle x_k - x_{k-1}, \hat x_{k+1}-x_k\rangle - \frac{1}{2\hat \alpha_k}\|\hat x_{k+1}-x_k\|^2
	+
	\langle \nabla f(\xi_k) - \nabla f(x_k) , \hat x_{k+1} - x_k\rangle \\
    &\le
	\frac{\beta_{\max}}{\alpha_{\min}}\| x_k - x_{k-1}\| \|\hat x_{k+1}-x_k\| - \frac{1}{2\hat \alpha_k}\|\hat x_{k+1}-x_k\|^2
	+
	\| \nabla f(\xi_k) - \nabla f(x_k) \| \|\hat x_{k+1} - x_k\|,
\end{align*}
where in the first inequality we used~\eqref{eqn:armijo_no}, in the first equality~\eqref{eqn:mean_value}, in the second inequality~\eqref{eqn:dis_prod_scalare}, and in the last one we employed~\eqref{eqn:beta/alpha} and  the Cauchy--Schwarz inequality. Rearranging this chain of inequalities and dividing by $\|\hat x_{k+1}-x_k\|$, we obtain
\begin{equation*}
\frac{1-c}{2\hat \alpha_k}\|\hat x_{k+1}-x_k\| \le \frac{\beta_{\max}}{\alpha_{\min}} \| x_k - x_{k-1}\|
	+
	\| \nabla f(\xi_k) - \nabla f(x_k) \| .
\end{equation*}
Thus, using~\eqref{eqn:conv_diff} and~\eqref{eqn:lim_diff_grad}, we obtain
\begin{equation}
\label{eqn:lim_hat/alpha}
\lim_{k \in K'', k\to \infty} \frac{\|\hat x_{k+1} - x_k\|}{\hat \alpha_k} = 0.
\end{equation}
Until this point, we focused on the relationship between $\hat x_{k+1}$ and $x_k$; now, we focus on the relationship between $\hat x_{k+1}$ and $x_{k+1}$. 
Since $\hat x_{k+1} \in \arg\min_{y \in D} Q_{k, j-1}(y)$, we have $Q_{k,j-1}(\hat x_{k+1}) \le Q_{k,j-1}(x_{k+1})$. Then, recalling~\eqref{eqn:Qkj}, we have
\begin{equation}
\label{eqn:dis_Q1}
\begin{split}
&\langle \nabla f(x_k) - \frac{\hat\beta_k}{\hat\alpha_k}(x_k-x_{k-1}), \hat x_{k+1}-x_k \rangle + \frac{1}{2\hat\alpha_k}\| \hat x_{k+1} - x_k\|^2 \\
\le  &\langle \nabla f(x_k) - \frac{\hat\beta_k}{\hat\alpha_k}(x_k-x_{k-1}), x_{k+1}-x_k \rangle + \frac{1}{2\hat\alpha_k}\|x_{k+1} - x_k\|^2.
\end{split}
\end{equation}
Analogously, since $ x_{k+1} \in \arg\min_{y \in D} Q_{k, j}(y)$, we have $Q_{k,j}( x_{k+1}) \le Q_{k,j}(\hat x_{k+1})$. Then, recalling~\eqref{eqn:Qkj}, we have
\begin{align*}
&\langle \nabla f(x_k) - \frac{\beta_k}{\alpha_k}(x_k-x_{k-1}), x_{k+1}-x_k \rangle + \frac{1}{2\alpha_k}\|  x_{k+1} - x_k\|^2 \\
\le  &\langle \nabla f(x_k) - \frac{\beta_k}{\alpha_k}(x_k-x_{k-1}), \hat x_{k+1}-x_k \rangle + \frac{1}{2\alpha_k}\|\hat x_{k+1} - x_k\|^2,
\end{align*}
which, remembering  that $\alpha_k= \delta\hat \alpha_k$ and $\beta_k=\delta^a \hat\beta_k$, can be rewritten as
\begin{equation}
\label{eqn:dis_Q2}
\begin{split}
&\langle \nabla f(x_k) - \frac{\delta^{a-1}\hat\beta_k}{\hat\alpha_k}(x_k-x_{k-1}), x_{k+1}-x_k \rangle + \frac{1}{2\delta\hat\alpha_k}\|  x_{k+1} - x_k\|^2 \\
\le &\langle \nabla f(x_k) - \frac{\delta^{a-1}\hat\beta_k}{\hat\alpha_k}(x_k-x_{k-1}), \hat x_{k+1}-x_k \rangle + \frac{1}{2\delta\hat\alpha_k}\|\hat x_{k+1} - x_k\|^2.
\end{split}
\end{equation}
Summing both sides of~\eqref{eqn:dis_Q1} and~\eqref{eqn:dis_Q2}, we have:
    \begin{equation}
    \label{eqn:dis_quadrati/alpha}
	 \frac{1-\delta}{2\delta\hat \alpha_k} \|x_{k+1}-x_k\|^2 \le \frac{1-\delta}{2\delta\hat \alpha_k} \|\hat x_{k+1}-x_k\|^2 + (1-\delta^{a-1})\frac{\hat \beta_k}{\hat \alpha_k}  
    \langle x_k - x_{k-1}, \hat x_{k+1} - x_{k+1} \rangle.
    \end{equation}
Hence, using $\alpha_k = \delta \hat \alpha_k$, and exploiting~\eqref{eqn:dis_quadrati/alpha} (after multiplying both sides by $\frac{2}{\delta(1-\delta)\hat \alpha_k}$), we obtain that for all $k \in K''$,
\begin{equation}
\label{eqn:dis_quadrati}
\begin{split}	
\frac{\|x_{k+1}-x_k\|^2}{\alpha_k^2} &= 
\frac{\|x_{k+1}-x_k\|^2}{\delta^2\hat \alpha_k^2} \\  &\le \frac{\|\hat x_{k+1}-x_k\|^2}{\delta^2\hat \alpha_k^2}  + \frac{2(1-\delta^{a-1}) }{\delta(1-\delta)}  
   \frac{\hat \beta_k}{\hat \alpha_k^2}  
    \langle x_k - x_{k-1}, \hat x_{k+1} - x_{k+1} \rangle \\
    & \le \frac{\|\hat x_{k+1}-x_k\|^2}{\delta^2\hat \alpha_k^2}  + \frac{2(1-\delta^{a-1}) }{\delta(1-\delta)} \frac{\beta_{\max}}{ \alpha_{\min}^2}  
    \| x_k - x_{k-1}\| \|\hat x_{k+1} - x_{k+1} \|,
\end{split}
\end{equation}
where in last inequality we used the Cauchy-Schwarz inequality and the fact that  
$$\frac{\hat \beta_k}{\hat \alpha_k^2}  = \frac{\beta_{k0}\delta^{(j-1)a}}{\alpha_{k0}^2(\delta^{j-1})^2} = \frac{\beta_{k0}}{\alpha_{k0}^2}\delta^{(j-1)(a-2)} \le \frac{\beta_{\max}}{\alpha_{\min}^2}.$$
Observe that, using~\eqref{eqn:conv_diff},~\eqref{eqn:lim_diff_xk+1}, and~\eqref{eqn:lim_hat/alpha}, the RHS of~\eqref{eqn:dis_quadrati} tends to $0$ for $k \in K''$ and $k\to\infty$.
Then, 
$$\lim_{k \in K'',k \to \infty} \frac{\|x_{k+1}-x_k\|}{\alpha_k}=0,$$
but this is in contrast with~\eqref{eqn:assurdo}. Hence, we conclude that~\eqref{eqn:tesi_conv_rapporto} is verified.
\end{proof}

\medskip
We are finally ready to state the first main asymptotic convergence result of this manuscript.

\medskip
\begin{theorem}
\label{theo:conv_eps0}
Let $a\ge2,\tau=0$. Assume that $f$ is continuously differentiable. Further assume
Assumption~\ref{ass:rule}. Then, if, for every $k$, $x_k$ is not an M-stationary point, then
any accumulation point  $\bar x$ of the sequence $\{x_k\}$ generated by Algorithm~\ref{alg:ggmm} is an M-stationary point for problem~\eqref{eqn:prob}.
\end{theorem}

\medskip
\begin{proof}
Let $\{x_k\}$ be the sequence generated by Algorithm~\ref{alg:ggmm} and 
assume that $\bar x \in D$ is an accumulation point, with $\{x_k\}_{k \in K}$ converging to $\bar x$ on a subsequence $K \subset \mathbb N$.
Remember, from~\eqref{eqn:M-subprob}, that $x_{k+1}$ satisfies 
$$
-\nabla f(x_k)-\frac{1}{\alpha_{k}}(x_{k+1}-x_k)+\frac{\beta_{k}}{\alpha_{k}}(x_k-x_{k-1}) \in \mathcal N_D(x_{k+1}).
$$
Due to Proposition~\ref{prop:conv_diff}, we also have that $ \{ x_{k+1}\}_{k \in K}$ converges to $\bar x$. From Proposition~\ref{prop:conv_rapporto}, we have that $ \|x_{k+1}-x_k\|/\alpha_k$ converges to $0$ for $k$ going to infinity and $k \in K$.
Hence, taking the limit for $k \in K, k \to \infty $
and exploiting the continuity of $\nabla f$, the upper semicontinuity of 
the general normal cone, and the fact that $\lim_{k \to \infty}\frac{\beta_k}{\alpha_k}(x_k-x_{k-1})=0$ (again from Proposition~\ref{prop:conv_diff} and the fact that $\frac{\beta_k}{\alpha_k} \in [0,\frac{\beta_{\max}}{\alpha_{\min}}]$), we obtain 
\begin{equation*}
   -\nabla f(\bar x) \in \mathcal N_D (\bar x),
\end{equation*}
i.e., $ \bar x $ is an M-stationary point for problem~\eqref{eqn:prob}.
\end{proof}

\subsection{Convergence to B- and P-stationary points for $\tau>0$}

Now, we assume that in Algorithm~\ref{alg:ggmm} we take $\tau>0$ and $a\ge1$ (possibly, $a=1$). We will show in this section that, similarly to the baseline PGM, if $\bar x$ is an accumulation point of the sequence $\{x_k\}$ the following statements hold:

\medskip
\begin{enumerate}
\item if $f$ is continuously differentiable and the stopping condition of Algorithm~\ref{alg:ggmm} is based on B-stationarity, then:
\smallskip
\begin{enumerate}[1)]
\item in a neighborhood of $\bar x$, the backtracking in Lines~\ref{line:linesearch_start}--\ref{line:linesearch_end} terminates in a bounded number of iterations (Propositions~\ref{prop:ArbitrarySmallAlpha}--\ref{prop:AcceptedStep});
\item $\bar x$ is a B-stationary point for problem~\eqref{eqn:prob} (Theorem~\ref{theo:conv_eps>0});
\end{enumerate}
\smallskip
\item if $f$ is continuously differentiable with $\nabla f$ locally Lipschitz and the stopping condition of Algorithm~\ref{alg:ggmm} is based on P-stationarity, then:
\smallskip
\begin{enumerate}[1)]
\item in a neighborhood of $\bar x$, the backtracking in Lines~\ref{line:linesearch_start}--\ref{line:linesearch_end} terminates in a bounded number of iterations (Proposition~\ref{prop:ArmijoConditionLipschitzContinuousGradient} and Corollary~\ref{coro:PGDmapTerminationLipschitzContinuousGradient});
\item $\bar x$ is a P-stationary point for problem~\eqref{eqn:prob} (Theorem~\ref{theo:conv_P}).
\end{enumerate}
\end{enumerate}

\medskip
We will treat in detail the two scenarios separately in the following subsections. 
\subsubsection{B-stationarity}
\label{sec:Bstat}
First, we only assume that $f$ is continuously differentiable and that the stopping condition of Algorithm~\ref{alg:ggmm} is based on B-stationarity, studying the convergence to B-stationary points. 
We begin by recalling a result from the literature.

\medskip
\begin{proposition}[{\cite[Prop.~5.2]{olikier2025projected}}]
\label{prop:ArbitrarySmallAlpha}
Assume that $f$ is continuously differentiable. Let $\bar x \in D$ be non-B-stationary for problem~\eqref{eqn:prob}, and $w \in T_D(\bar x)$ be such that
$$
\langle w, \nabla f(\bar x) \rangle < 0.
$$
Let $\delta \in (0, 1)$. Define $\kappa := \sqrt{1 - \frac{\delta \langle w,\nabla f(\bar x)\rangle^2} {8 \|w\|^2 \|\nabla f(\bar x)\|^2}} \in (0, 1)$.
For every $\varepsilon \in (0, \infty)$, there exist $\alpha_{\bar x} \in (0, \varepsilon]$ and $\bar{\rho}(\alpha_{\bar x}) \in (0, \infty)$ such that, for all $x \in B(\bar x, \bar{\rho}(\alpha_{\bar x})) \cap D$ and $\alpha \in [\alpha_{\bar x}, \alpha_{\bar x}/\delta]$,
\begin{equation*}
d_D(x-\alpha\nabla f(x)) \le \kappa \alpha \|\nabla f(x)\|,
\end{equation*}
which implies, for all $y \in P_D(x-\alpha\nabla f(x))$,
\begin{equation*}
\langle \nabla f(x),y-x\rangle \le - \sqrt{1-\kappa^2} \|\nabla f(x)\| \|y-x\|.
\end{equation*}
\end{proposition}

At this point, we are able to prove that in a neighborhood of a non-B-stationary point $\bar x$ the backtracking procedure terminates in a bounded number of iterations. Specifically, we show that there exists an interval of stepsizes $[\alpha_{\bar x}, \alpha_{\bar x}/\delta]$ (for some $\alpha_{\bar x} >0$) where the sufficient decrease condition holds and which is certainly visited throughout the backtracking procedure.

\medskip
\begin{proposition}
\label{prop:AcceptedStep}
Let $0 < \tau < \alpha_{\min}\le \alpha_{\max}<\infty$, $a\ge1,\beta_{\max}\ge0$,  and $\delta, c \in (0, 1)$.
Assume that $f$ is continuously differentiable. Let $\bar x \in D$ be non-B-stationary for problem~\eqref{eqn:prob}.
There exist $\alpha_{\bar x} \in (0, \delta\tau]$ and $\rho \in (0, \infty)$ such that, for all $x \in B(\bar x, \rho) \cap D$, $\alpha_0\in[ \alpha_{\min},  \alpha_{\max}]$, $\beta_0\in[0,\beta_{\max}]$, $s \in \mathbb X$,  $\alpha \in [\alpha_{\bar x}, \alpha_{\bar x}/\delta]$, and $y \in P_D(x-\alpha\nabla f(x)+\beta(\alpha) s)$, where $\beta(\alpha) = \beta_0 \big(\frac{\max\{\alpha-\tau,0\}}{\alpha_0-\tau}\big)^a$, the following sufficient decrease condition holds:
\begin{equation*}
f(y) < f(x) - \frac{c}{2\alpha}\|y-x\|^2.
\end{equation*}
\end{proposition}

\medskip
\begin{proof}
Define $\kappa$ as in Proposition~\ref{prop:ArbitrarySmallAlpha}. Let $\xi \in (0, \infty)$ be small enough to ensure
\begin{subequations}
\begin{align*}
\sup_{y \in  B\left(\bar x, \frac{7\xi}{2\delta}\|\nabla f(\bar x)\|\right) \cap D \setminus \{\bar x\}}
\frac{|f(y)-f(\bar x) - \langle\nabla f(\bar x),y-\bar x\rangle|}{\|y-\bar x\|}
&< \frac{(1-c)\sqrt{1-\kappa^2}\|\nabla f(\bar x)\|}{4\left(1 + \frac{8}{3(1-\kappa)}\right)},\\
\sup_{y \in  B\left(\bar x, \frac{7\xi}{2\delta}\|\nabla f(\bar x)\|\right) \cap D}
\|\nabla f(y)-\nabla f(\bar x)\|
&< \frac{(1-c)\sqrt{1-\kappa^2}}{4} \|\nabla f(\bar x)\|.
\end{align*}
\end{subequations}
Such a value $\xi$ certainly exists by the definition and the continuity of the gradient at $\bar{x}$.
Then, we can define $\varepsilon \coloneq \min\left\{ \xi,\delta\tau\right\}$, let $\alpha_{\bar x} \in (0, \varepsilon]$ and choose $\bar{\rho}(\alpha_{\bar x})\in(0,\infty)$ as given in Proposition~\ref{prop:ArbitrarySmallAlpha}. 

Since $\alpha_{\bar x} \le \varepsilon \le \delta\tau$, we can consequently observe that, for every $\alpha \in [\alpha_{\bar x}, \alpha_{\bar x}/\delta]$, we have $\alpha \le \alpha_{\bar x}/\delta\le\tau$, and hence $\max\{\alpha-\tau,0\}=0$.  For all $\alpha_0\in[\alpha_{\min}, \alpha_{\max}]$, $\beta_0\in[0,\beta_{\max}]$, and $a\ge1$, we therefore have 
$\beta(\alpha) = \beta_0 \big(\frac{\max\{\alpha-\tau,0\}}{\alpha_0-\tau}\big)^a =  0$. 
Then, for all $s \in \mathbb X$,  $P_D(x-\alpha\nabla f(x)+\beta(\alpha) s)=P_D(x-\alpha\nabla f(x))$. 
The remainder of the proof follows exactly as in~\cite[Prop.~5.3]{olikier2025projected}; for completeness, we report it in its entirety in Appendix \ref{sec:appendix}. 
\end{proof}

\medskip
We finally turn to the main result of convergence towards B-stationary points.

\medskip
\begin{theorem}
\label{theo:conv_eps>0}
Let $a\ge1,\tau>0$. Assume that the stopping condition of Algorithm~\ref{alg:ggmm} is based on B-stationarity. Further assume Assumption~\ref{ass:rule}. If, for every $k$, $x_k$ is not a B-stationary point, then
any accumulation point  $\bar x$ of the sequence $\{x_k\}$ generated by  Algorithm~\ref{alg:ggmm} is a B-stationary point for problem~\eqref{eqn:prob}. 
\end{theorem}

\medskip
\begin{proof}
Let $\bar x\in D$ be an accumulation point of $\{x_k\}$, with $\{x_k\}_{k \in K}$ converging to $\bar x$ on a subsequence $K \subset \mathbb N$. 
Let us assume by contradiction that $\bar x$ is not B-stationary for problem~\eqref{eqn:prob}. 
Let $\alpha_{\bar x}$ and $\rho$ be as in Proposition~\ref{prop:AcceptedStep}. For all $k \in K$ large enough, $x_k \! \in B(\bar x, \rho) \cap D$ ($x_k$ belongs to $D$ by the instructions of Algorithm~\ref{alg:ggmm}). Thus, the backtracking procedure in Lines~\ref{line:linesearch_start}--\ref{line:linesearch_end} terminates  with some $\alpha_k \ge \alpha_{\bar x}$ that satisfies the sufficient decrease condition, meaning that
\begin{equation*}
x_{k+1} \in P_D(x_k-\alpha_k \nabla f(x_k) +\beta_k(x_k-x_{k-1}))
\text{ for some } \alpha_k \in [\alpha_{\bar x}, \alpha_{\max}] \text{ and } \beta_k\in[0,\beta_{\max}].
\end{equation*}
We can assume (possibly replacing $K$ with a further subsequence) that $\{\alpha_k\}_{k\in K}$ converges to  $\alpha_{\lim} \in [\alpha_{\bar x}, \alpha_{\max}]$ and $\{\beta_k\}_{k\in K}$ converges to $\beta_{\lim} \in [0, \beta_{\max}]$.
For all $k \in K$, we have
\begin{equation*}
\|x_k - \alpha_k \nabla f(x_k) +\beta_k(x_k-x_{k-1}) - x_{k+1}\|
= d_D(x_k - \alpha_k \nabla f(x_k) +\beta_k(x_k-x_{k-1})),
\end{equation*}
and, by the continuity of the distance  to a nonempty closed set, take the limits: using $\lim_{k \to \infty}x_{k+1} =\bar x$ and $\lim_{k \to \infty} x_{k-1} = \bar x$, from Proposition~\ref{prop:conv_diff}, yields
\begin{equation*}
\alpha_{\lim}\| \nabla f(\bar x)\|
= d_D(\bar x - \alpha_{\lim}\nabla f(\bar x)).
\end{equation*}
Recalling Remark~\ref{obs:equiv_prox}, this is equivalent to $\bar x \in P_D(\bar x - \alpha_{\lim}\nabla f(\bar x))$, which in turn implies $-\nabla f(\bar x) \in \widehat{\mathcal N}_D^P(\bar x) \subseteq \widehat{\mathcal N}_D(\bar x)$. This contradicts the fact that $\bar x$ is not B-stationary for problem~\eqref{eqn:prob}. Hence, every accumulation point has to be B-stationary.
\end{proof}

\medskip
\begin{remark}
\label{obs:fail_prop}
In the proof of the previous proposition, we started by assuming that $\bar x$ is not a B-stationary point and we came to the contradiction that $\bar x$ is P-stationary. One may think that with the same reasoning the convergence to P-stationary points could be proven. However, even if $ \bar x$ is not P-stationary, it can be a B-stationary point. In this case, there does not exist any $w \in T_D(\bar x)$ such that
$\langle w, \nabla f(\bar x) \rangle < 0$. Hence, Proposition~\ref{prop:ArbitrarySmallAlpha} would not be exploitable in the reasoning.
\end{remark}

\subsubsection{P-stationarity}
\label{sec:Pstat}
We finally make the further assumption that $\nabla f$ is locally Lipschitz continuous and that the stopping condition of Algorithm~\ref{alg:ggmm} is based on P-stationarity, studying the convergence to P-stationary points.
As highlighted in Remark~\ref{obs:fail_prop}, the reasoning undertaken in the previous subsection cannot be repeated as a point can be non-P-stationary but B-stationary. However, exploiting the local Lipschitz property, we can still find an interval of stepsizes where the sufficient decrease condition is satisfied. Remember that, using the same notation as in~\cite{olikier2025projected},  for every ball $B(\bar x,\bar\rho)$,
\[\text{Lip}_{B(\bar x,\bar\rho)}(\nabla f):=\sup_{x,y \in B(\bar x,\bar\rho), x\not=y} \frac{\|\nabla f(x)-\nabla f(y)\|}{\|x-y\|} < \infty.\]
This implies (see, e.g.,~\cite{bertsekas1995nonlinear,grippo2023introduction}) that for every $x,y \in B(\bar x, \bar \rho)$,
\begin{equation}
\label{eqn:descentlemma}
|f(y)-f(x)- \langle \nabla f(x),y-x\rangle| \le \frac{\text{Lip}_{B(\bar x,\bar\rho)}(\nabla f)}{2}\|y-x\|^2.
\end{equation}

\begin{proposition}
\label{prop:ArmijoConditionLipschitzContinuousGradient}
Let $0 < \tau < \alpha_{\min}\le \alpha_{\max}<\infty$, $a\ge1,\beta_{\max}\ge0$,  $ c \in (0, 1)$, and $\rho \in (0,\infty)$.
Assume that $f$ is continuously differentiable with $\nabla f$ locally Lipschitz continuous. Let $\bar x \in D$, and $\bar{\rho} \in \left[\rho + 2 \alpha_{\max} \max_{x \in B (\bar x, \rho) \cap D} \|\nabla f(x)\|, \infty\right)$. Define $\alpha_* \coloneq (1-c)/\textnormal{Lip}_{B(\bar x, \bar{\rho})}(\nabla f)$.
Then, for all $x \in B(\bar x, \rho) \cap D$, $\alpha_0\in[ \alpha_{\min},  \alpha_{\max}]$, $\beta_0\in[0,\beta_{\max}]$, $s \in \mathbb X$,  $\alpha \in (0, \min\{\alpha_*, \tau\}]$, and $y \in P_D(x-\alpha\nabla f(x)+\beta(\alpha) s)$, where $\beta(\alpha) = \beta_0 \big(\frac{\max\{\alpha-\tau,0\}}{\alpha_0-\tau}\big)^a$,
\begin{equation*}
f(y) \le f(x) - \frac{c}{2\alpha}\|y-x\|^2.
\end{equation*}
\end{proposition}

\begin{proof}
For all $x \in B(\bar x, \rho) \cap D$ and $\alpha \in (0, \min\{\alpha_*, \tau\}]$, $P_D(x-\alpha\nabla f(x)) \subseteq B(\bar x, \bar{\rho})$; indeed, also recalling~\eqref{eqn:propr_proj1}, for all $y \in P_D(x-\alpha\nabla f(x))$ we have
\begin{equation*}
\|y-\bar x\|
\le \|y-x\| + \|x-\bar x\|
\le 2 \alpha \|\nabla f(x)\| + \rho
\le \bar{\rho},
\end{equation*}
Since $\alpha \le \tau$, $\max\{\alpha-\tau,0\}=0$. As a consequence, for all $\alpha_0\in[\alpha_{\min}, \alpha_{\max}]$, $\beta_0\in[0,\beta_{\max}]$, and $a\ge1$, 
$\beta(\alpha) = \beta_0 \big(\frac{\max\{\alpha-\tau,0\}}{\alpha_0-\tau}\big)^a =  0$. Then, for all $s \in \mathbb X$,  $P_D(x-\alpha\nabla f(x)+\beta(\alpha) s)=P_D(x-\alpha\nabla f(x))$.
Thus, for all $x \in B(\bar x, \rho) \cap D$, $\alpha \in (0, \min\{\alpha_*, \tau\}]$, and $y \in P_D(x-\alpha\nabla f(x)+\beta(\alpha) s)=P_D(x-\alpha\nabla f(x))$,
\begin{align*}
f(y)
&\le f(x) + \langle \nabla f(x), y-x \rangle + \frac{1}{2} \text{Lip}_{B(\bar x, \bar{\rho})}(\nabla f) \|y-x\|^2\\
&\le f(x) + \frac{1}{2\alpha}\left(-1+\alpha\text{Lip}_{B(\bar x, \bar{\rho})}(\nabla f)\right) \|y-x\|^2\\
& \le f(x) - \frac{c}{2\alpha}\|y-x\|^2, 
\end{align*}
where the first inequality comes from~\eqref{eqn:descentlemma}, the second from~\eqref{eqn:propr_proj2}, the third from $\alpha \le \alpha_*$ and the definition of $\alpha_*$.
\end{proof}

\medskip
Now we can conclude that, for all points in a neighborhood of a given solution $\bar x$, the backtracking procedure terminates in a number of iterations which is bounded above by a constant (that may depend on $\bar{x}$ itself). 

\medskip
\begin{corollary}
\label{coro:PGDmapTerminationLipschitzContinuousGradient}
Consider Algorithm~\ref{alg:ggmm} under $a \ge 1$, $\tau>0$ and assume that the stopping condition of Algorithm~\ref{alg:ggmm} is based on P-stationarity. Assume $f$ is continuously differentiable with $\nabla f$ locally Lipschitz. Given $\bar x \in D$ and $\rho \in (0, \infty)$, let $\bar \rho$ and $\alpha_*$ be as in Proposition~\ref{prop:ArmijoConditionLipschitzContinuousGradient}.
Then, for every $x \in B(\bar x, \rho) \cap D$, the backtracking in Lines~\ref{line:linesearch_start}--\ref{line:linesearch_end} terminates with a step size $\alpha \in [\delta\min\{\tau, \alpha_*\}, \alpha_{\max}]$ and hence after at most $\max\{0, \lceil\log_{1/\delta}(\min\{\tau,\alpha_*\}/\alpha_0)\rceil\}$ iterations, where $\alpha_0\in[\alpha_{\min},\alpha_{\max}]$ is the initial step size chosen in Line~\ref{line:initialparameters}.
\end{corollary}

\medskip
\begin{proof}
Let $\bar j \in \mathbb N$ be the largest integer $j$ such that, defined $\alpha = \alpha_0 \delta^j$, $\alpha/\delta > \min\{\tau,\alpha_*\}$. Observe that  $\bar j < 1+\log_{1/\delta}(\min\{\tau,\alpha_*\}/\alpha_0)$.  By Proposition~\ref{prop:ArmijoConditionLipschitzContinuousGradient},  the while loop in Lines~\ref{line:linesearch_start}--\ref{line:linesearch_end} terminates after at most $\bar j$ iterations,  i.e., after at most $\max\{0, \lceil\log_{1 /\delta}(\min\{\tau,\alpha_*\}/\alpha_0)\rceil\}$ iterations.
\end{proof}

\medskip
The previous results allow us to finally retrieve the last convergence result of this paper.

\medskip
\begin{theorem}
\label{theo:conv_P}
Let $a\ge1,\tau>0$. Assume that the stopping condition of Algorithm~\ref{alg:ggmm} is based on P-stationarity.  Assume that $f$ is continuously differentiable with $\nabla f$ locally Lipschitz. Further assume
Assumption~\ref{ass:rule}. Then, if $x_k$ is not a P-stationary point for every $k$, then
any accumulation point  $\bar x$ of the sequence $\{x_k\}$ generated by  Algorithm~\ref{alg:ggmm} is a P-stationary point for problem~\eqref{eqn:prob}. Moreover, for any converging subsequence $\{x_k\}_{k \in K}$, 
\begin{equation}
\label{eqn:conv_dist}
\lim_{k \in K, k \to \infty} d_{\widehat{ \mathcal N}_D^P(x_{k+1})}(- \nabla f(x_{k+1}))=0.
\end{equation}

\end{theorem}

\begin{proof}
Let $\bar x\in D$ be an accumulation point of $\{x_k\}$. Let $K \subset \mathbb N$ such that $x_k\to \bar{x}$ for $k\in K$, $k\to\infty$.  Given $\rho \in (0, \infty)$, let $\bar{\rho}$ and $\alpha_*$ be as in Proposition~\ref{prop:ArmijoConditionLipschitzContinuousGradient} and also let
$I \coloneq \left[\delta\min\{{\tau }, \alpha_*\}, \alpha_{\max}\right]$.
There exists $k_* \in \N$ such that, for all integers $k \ge k_*, k \in K$, $x_k \in B(\bar x, \rho)$, thus, by Corollary~\ref{coro:PGDmapTerminationLipschitzContinuousGradient}, $x_{k+1} \in P_D(x_k-\alpha_k\nabla f(x_k)+\beta_k(x_k-x_{k-1}))$ with $\alpha_k \in I$ and $\beta_k \in [0,\beta_{\max}]$, and hence
\begin{equation}
\label{eqn:dist_D}
\|x_{k+1}-(x_k-\alpha_k\nabla f(x_k)+\beta_k(x_k-x_{k-1}))\| = d_D(x_k-\alpha_k\nabla f(x_k)+\beta_k(x_k-x_{k-1})).
\end{equation}
Since $I$ and $[0,\beta_{\max}]$ are compact, there exist an infinite subset $K'\subset K$, and two subsequences $\{\alpha_k\}_{k \in K'}$ and $\{\beta_k\}_{k \in K'}$ converging to $\bar\alpha \in I$ and $\bar\beta \in [0,\beta_{\max}]$, respectively. Furthermore, by Proposition~\ref{prop:conv_diff}, $\{x_{k+1}\}_{k \in K'}$ and $\{x_{k-1}\}_{k \in K'}$ converge to $\bar x$. Therefore, taking the limit for $k \to \infty$, with $k \in K'$, in~\eqref{eqn:dist_D}, we obtain
\begin{equation*}
\|\bar x-(\bar x-\bar\alpha\nabla f(\bar x)+\bar\beta (\bar x-\bar x))\| = d_D(\bar x-\bar\alpha\nabla f(\bar x)+\bar\beta (\bar x-\bar x)),
\end{equation*}
i.e.,
\begin{equation*}
\bar\alpha\|\nabla f(\bar x)\| = d_D(\bar x-\bar\alpha\nabla f(\bar x)).
\end{equation*}
It follows that $\bar x \in P_D(\bar x-\bar\alpha\nabla f(\bar x))$, which implies, from Remark~\ref{obs:equiv_prox},  that $-\nabla f(\bar x) \in \widehat{\mathcal N}_D^P(\bar x)$.

We now prove~\eqref{eqn:conv_dist}. Remember that for all $k \ge k_*$, $x_{k+1} \in P_D(x_k-\alpha_k\nabla f(x_k)+\beta_k(x_k-x_{k-1}))$, with $\alpha_k \in I$, and $\beta_k \in [0, \beta_{\max}]$. Then, $\frac{1}{\alpha_k}(x_k-x_{k+1})+\frac{\beta_k}{\alpha_k}(x_k-x_{k-1})- \nabla f(x_k) \in \widehat{\mathcal N}_D^P(x_{k+1})$. Thus,
\begin{align*}
d_{\widehat{ \mathcal N}_D^P(x_{k+1})}(- \nabla f(x_{k+1}))
&\le \big\| - \nabla f(x_{k+1})- \big( \frac{1}{\alpha_k}(x_k-x_{k+1})+\frac{\beta_k}{\alpha_k}(x_k-x_{k-1})- \nabla f(x_k) \big) \big\|\\
&\le \frac{1}{\alpha_k}\|x_k- x_{k+1}\| + \frac{\beta_k}{\alpha_k}\|x_k- x_{k-1}\| + \|\nabla f(x_{k+1}) - \nabla f(x_k) \|, 
\end{align*}
which converges to $0$ for $k \in K$ and $k \to \infty$ since $\{x_k\}_{k \in K}$, $\{x_{k+1}\}_{k \in K}$ and $\{x_{k-1}\}_{k \in K}$ converge to $ \bar x$, $\{\alpha_k\}_{k \in K}$ is bounded away from zero, and $\beta_k\in [0,\beta_{\max}]$ for all $k$.
\end{proof}

\medskip
As in~\cite{olikier2023first}, we consequently obtain, for bounded sequences $\{x_k\}$, the property stated in the following corollary. The boundedness assumption is easily verified, e.g., if the sublevel set $\mathcal L_0=\{x \in D \mid f(x) \le f(x_0)\}$ is bounded from below.

\medskip
\begin{corollary}
\label{cor:conv_dist}
Let $a\ge1,\tau>0$. Assume that the stopping condition of Algorithm~\ref{alg:ggmm} is based on P-stationarity.  Assume that $f$ is continuously differentiable with $\nabla f$ locally Lipschitz. Further assume
Assumption~\ref{ass:rule}. 
If the sequence $\{x_k\}$ generated by Algorithm~\ref{alg:ggmm} is bounded and, for every $k$, $x_k$ is not a P-stationary point, then all accumulation points of $\{x_k\}$ are P-stationary points for problem~\eqref{eqn:prob}. Moreover, they attain the same value of $f$ and
\begin{equation}
\label{eqn:conv_dist_cor}
\lim_{k \to \infty} d_{\widehat{ \mathcal N}_D^P(x_{k})}(- \nabla f(x_{k}))=0.
\end{equation}
\end{corollary}

\begin{proof}
The proof is the same as in~\cite[Prop.~6.4]{olikier2023first}.
\end{proof}

\medskip
Thanks to~\eqref{eqn:conv_dist_cor}, we are guaranteed that, as long as the sequence ${x_k}$ remains bounded, for every $\epsilon>0$ there exists $\bar{k}\in\mathbb{N}$ such that 
$$d_{\widehat{ \mathcal N}_D^P(x_{k})}(- \nabla f(x_{k})) < \epsilon$$
for all $k\geq \bar{k}$. This property justifies the use of this condition as a reliable stopping criterion for the algorithm.
In particular, it follows that the accumulation points cannot be \emph{serendipitous}, according to the definition in~\cite{levin2023finding} (see also~\cite{olikier2026lowrank}). Moreover, Theorem~\ref{theo:conv_eps>0} implies that no accumulation point $\bar{x}$ can be \emph{apocalyptic}, that is, there cannot exist a subsequence $\{x_k\}_{k\in K}$ converging to $\bar{x}$ such that $d_{\widehat{ \mathcal N}_D(x_{k})}
(- \nabla f(x_{k}))$ tends to $0$ on the subsequence $K$, while $d_{\widehat{ \mathcal N}_D(\bar x)}(- \nabla f(\bar x))>0$. Equivalently, no accumulation point can be M-stationary without also being B-stationary.

\section{Numerical Results}
\label{sec:experiments}

In this section, we present the results of computational experiments performed to assess the performance of the proposed approach.

The experimental analysis is divided into two parts, corresponding to two classes of optimization problems: cardinality-constrained problems and low-rank matrix completion problems. The results for each class are discussed separately in the following subsections.

Despite addressing different problem classes, both sets of experiments follow the same general experimental setup. We compare the proposed solver, the Geometric Gradient Method with Momentum (GGMM), with the (spectral) Projected Gradient Method (PGM)~\cite{jia2023augmented,olikier2025projected}. PGM serves as a natural baseline, as it relies solely on projected gradient directions without incorporating momentum or preconditioning mechanisms. In addition, we compare GGMM with a modified version of PANOC+ (P+)~\cite{de2022proximal}, in which the quasi-Newton search direction is replaced by a momentum-based direction, as described in Section~\ref{sec:panoc}.

The implementation of the proposed algorithms, together with the code used to reproduce the experiments, is publicly available at \href{https://github.com/DiegoScuppa/GGMM}{\tt https://github.com/DiegoScuppa/GGMM}. All experiments are conducted in a Python 3.13.5 environment on a machine equipped with an Intel(R) Xeon(R) Gold 5218 CPU running at 2.30~GHz, featuring 8 physical cores and 96~GB of RAM.

Results are presented in the form of performance profiles~\cite{dolan2004performance}. Specifically, we compare the three solvers on the following metrics:
\smallskip
\begin{enumerate}
    \item CPU time;
    \item Number of (outer) iterations;
    \item Number of function evaluations (which is equivalent to the number of inner iterations for GGMM and PGM, and to the double of this number for P+);
    \item Cumulative distribution of the relative gap of the final function value with respect to the best value found by any of the solvers.
\end{enumerate}
\smallskip
Complete results are available in the GitHub repository.

\subsection{Implementation details}

The implementation of GGMM follows the pseudocode reported in Algorithm~\ref{alg:ggmm}. We now discuss some implementation details. In particular, the pseudocode does not specify how the initial coefficients $\alpha_{k0}$ and $\beta_{k0}$ are chosen at each iteration $k$. In our implementation, these parameters are selected as follows.

Let $0 < \alpha_{\min} \leq \alpha_{\max} < \infty$, $\theta_{\min}\in[0,1]$, and $\tilde\beta_{\max}\geq 0$. At iteration $k>0$, define
$s_k=x_k-x_{k-1}$, and  $y_k=\nabla f(x_k)-\nabla f(x_{k-1})$.
If $\langle s_k,y_k\rangle>0$, the initial value of $\alpha_{k0}$ is chosen as 
$$
\alpha_{k0}=\max\bigg\{\alpha_{\min},\min\bigg\{\alpha_{\max},\frac{\|s_k\|^2}{\langle s_k,y_k\rangle}\bigg\}\bigg\},
$$
while $\beta_{k0}$ is computed as 
\begin{equation}
\label{eqn:beta_k}
\begin{split}
\beta_{k0}
=&\alpha_{k0}
\cdot\max\bigg\{\theta_{\min},
-\frac{\langle\nabla f(x_k),s_k\rangle}
{\|\nabla f(x_k)\|\|s_k\|}\bigg\} \\
&\cdot
\max\bigg\{0,\min\bigg\{\tilde\beta_{\max},
\frac{\langle\nabla f(x_k),y_k\rangle}{\langle s_k,y_k\rangle}
-2\frac{\|y_k\|^2\langle\nabla f(x_k),s_k\rangle}
{\langle s_k,y_k\rangle^2}
\bigg\}\bigg\}.
\end{split}
\end{equation}

Thus, $\alpha_{k0}$ is initialized using the safeguarded Barzilai--Borwein spectral steplength~\cite{barzilai1988twopoint}, whereas $\beta_{k0}$ is obtained by scaling $\alpha_{k0}$ by two factors. The first factor is a safeguarded measure of the alignment between the search direction $s_k$ and the negative gradient $-\nabla f(x_k)$, whereas the second is a safeguarded quantity inspired by the conjugate gradient parameter proposed by Hager and Zhang~\cite{hager2005cg}. If $\langle s_k, y_k \rangle \le 0$, $\alpha_{k0}$ and $\beta_{k0}$ are set to $1$ and $0$, respectively, as well as for iteration $k=0$. We observe that, in practice, this inner product is  rarely non-positive.

The implementation of PGM is identical to that of GGMM, except that $\beta_{k0}$ is fixed to zero at every iteration.

The parameters are set to
$\alpha_{\min}=10^{-6}$, $\alpha_{\max}=10^{6}$, $\tilde\beta_{\max}=10^{6}$, and $\theta_{\min}=10^{-2}$.
The remaining parameters of Algorithm~\ref{alg:ggmm} are chosen as
$a=1$, $\tau=10^{-7}$, $\delta=0.5$, and $c=10^{-5}$. 

A nonmonotone backtracking  based on the \textit{max-rule} is employed, with memory parameter $M=10$. Problem-dependent stopping criteria are adopted for the two classes of optimization problems considered. In addition, a time limit of $600$ seconds and a maximum of $10^9$ iterations are enforced as secondary stopping criteria in all experiments, unless otherwise stated.

\smallskip
\begin{remark}
This parameter setup (in particular, $\tau=10^{-7}$) for GGMM leads to using exact PGM steps only for extremely small stepsizes, so that the momentum term is usually actually exploited while we retain the stronger theoretical convergence properties. We observed that setups with $\tau=0$ lead to comparable results in practice; then, an open question remains regarding the convergence to B-stationarity for the case where alignment of the steps with the negative gradient occurs only asymptotically.    
\end{remark}

 \subsection{Globalization of Projected Gradient Method with Momentum via PANOC+ framework}
\label{sec:panoc}
To compare GGMM with an algorithm that explicitly exploits momentum, we analyze as an alternative a version of the momentum method globalized by means of the P+ framework (although only convergence to M-stationarity is guaranteed in this case). We refer the reader to~\cite[Algorithm~2]{de2022proximal} for the general P+ framework. Briefly, the method proceeds as follows: at iteration $k$, the algorithm maintains an infeasible point $x_k$ together with the projected feasible point
$$
\bar x_k \in P_D(x_k-\gamma_k\nabla f(x_k)) \subseteq D.
$$
The next iterate is then computed as
$$
x_{k+1}=(1-\tau_{k+1})\bar x_k+\tau_{k+1}(x_k+d_{k+1}),
$$
where $d_{k+1}\in\mathbb{R}^n$ is a search direction and $\tau_{k+1}\in(0,1]$ is determined by a backtracking procedure. As $\tau_{k+1}\to0$, the update reduces to the projected gradient step,
$$
x_{k+1}\in P_D(x_k-\gamma_k\nabla f(x_k)).
$$
Moreover, the steplength parameter $\gamma_{k}\in(0,\gamma_0]$ is updated through a second backtracking procedure and may be halved from one iteration to the next.

To isolate the effect of momentum within the PANOC+ framework, we define the search direction as
$$
d_{k+1}=\bar x_k-x_k+\tilde\beta_k(x_k-x_{k-1}),
$$
where the momentum parameter $\tilde\beta_k$ is computed analogously to~\eqref{eqn:beta_k}. With this choice, the update becomes
$$
x_{k+1} \in  P_D(x_k-\gamma_k\nabla f(x_k))
+\tau_{k+1}\tilde\beta_k(x_k-x_{k-1}),
$$
which resembles the GGMM iteration, except that the momentum term is applied after the projection rather than being incorporated within it. We remark that the resulting direction is not necessarily feasible; however, this does not represent an issue, since P+ does not require the sequence of iterates $\{x_k\}$ to remain feasible. Feasibility is instead enforced through the projected point $\bar x_k$, which is the point used to evaluate the stationarity measure by means of the same stopping criteria adopted for GGMM and PGM.  

The original Julia implementation\footnote{available at  \href{https://github.com/JuliaFirstOrder/ProximalAlgorithms.jl}{\tt https://github.com/JuliaFirstOrder/ProximalAlgorithms.jl}}  was adapted to Python and modified to incorporate the proposed momentum-based direction while preserving the original algorithmic setup. In particular, the initial steplength $\gamma_0$ is chosen as an approximation of the inverse Lipschitz constant, and the algorithm terminates whenever the backtracking procedure reduces the steplength $\gamma_k$ below $10^{-7}$.

\subsection{Cardinality-constrained problems}

For the first class of test problems, we consider cardinality-constrained optimization problems. Specifically, we study problem~\eqref{eqn:prob} with $\mathbb{X}=\mathbb{R}^n$, feasible set
$$
D=\{x\in\mathbb{R}^n:\|x\|_0\le s\},
$$
for some $s\in\{1,\ldots,n\}$, and objective function $f$ given by one of the following:
\smallskip
\begin{enumerate}[a)]
\item a quadratic function
$$
f(x)=x^\top Qx+c^\top x,
$$
where $Q\in\mathbb{R}^{n\times n}$ is positive semidefinite and $c\in\mathbb{R}^n$;

\item the logistic regression loss
$$
f(x)=\frac{1}{m}\sum_{i=1}^m\log\left(1+\exp\left(-y_i(Mx)_i\right)\right),
$$
where $M\in\mathbb{R}^{m\times n}$ is the data matrix and $y\in\{-1,1\}^m$ is the vector of class labels.
\end{enumerate}
\smallskip
For this class of problems, we adopt the stopping criterion
$$
\|\nabla_S f(x)\|<\texttt{tol},\quad \text{where}\quad S=
\begin{cases}
\{1,\ldots,n\}, & \text{if }\|x\|_0<s, \\
\{i\in\{1,\ldots,n\}:x_i\neq0\}, & \text{if }\|x\|_0=s.
\end{cases}
$$
which corresponds to the (approximate) Basic Feasibility condition discussed in Example~\ref{ex:sparsity} (see also~\cite{beck2013sparsity}), i.e., we are considering B/P-stationarity, which coincide for this problem (see, e.g.,~\cite[Prop.~7.1]{olikier2025projected}). 
In all experiments, we set $\texttt{tol}=10^{-4}$, and we choose the initial point $x_0$ as the null vector.

\subsubsection{Quadratic problems}
We now present the numerical results for the quadratic test problems. We consider five values of $n\in\{50,100,200,500,1000\}$ and six values of $\mathrm{cond}\in\{10,100,500,1000,5000,10000\}$. For each combination of $(n,\mathrm{cond})$, we generate 10 random positive semidefinite matrices $Q\in\mathbb{R}^{n\times n}$ with condition number $\mathrm{cond}$, together with random vectors $c\in\mathbb{R}^n$, yielding a total of 300 test instances. In all experiments, the cardinality parameter is set to $s=n/10$.

The corresponding performance profiles are reported in Figure~\ref{fig:pp_QP}. All instances are solved successfully by both GGMM and the momentum-based variant of P+, whereas PGM fails to solve one instance within the prescribed time limit. P+ achieves the best CPU time on the largest fraction of instances ($55\%$), followed by GGMM ($33\%$) and PGM ($12\%$). In terms of iteration count, GGMM and P+ exhibit comparable performance and both outperform PGM. On the other hand, P+ requires fewer function evaluations than the other methods. Finally, GGMM consistently attains the lowest objective values, whereas P+ exhibits noticeably poorer performance in this respect. This suggests that the shorter CPU times of P+ are largely due to its tendency to terminate at solutions with higher objective values.

\begin{figure}[htbp]
  \centering
  
  \begin{subfigure}{0.475\textwidth}
    \centering
    \includegraphics[width=\textwidth]{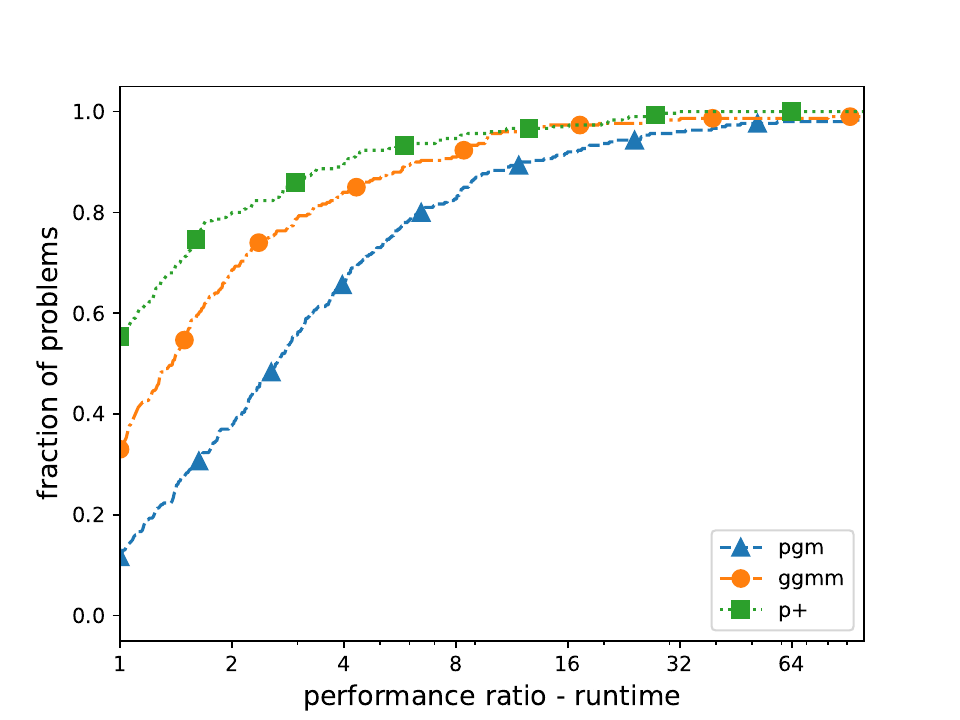}
  \caption{CPU time}
  \label{fig:pp_time_QP}
    \end{subfigure}
\hfill
\begin{subfigure}{0.475\textwidth}
    \centering
    \includegraphics[width=\textwidth]{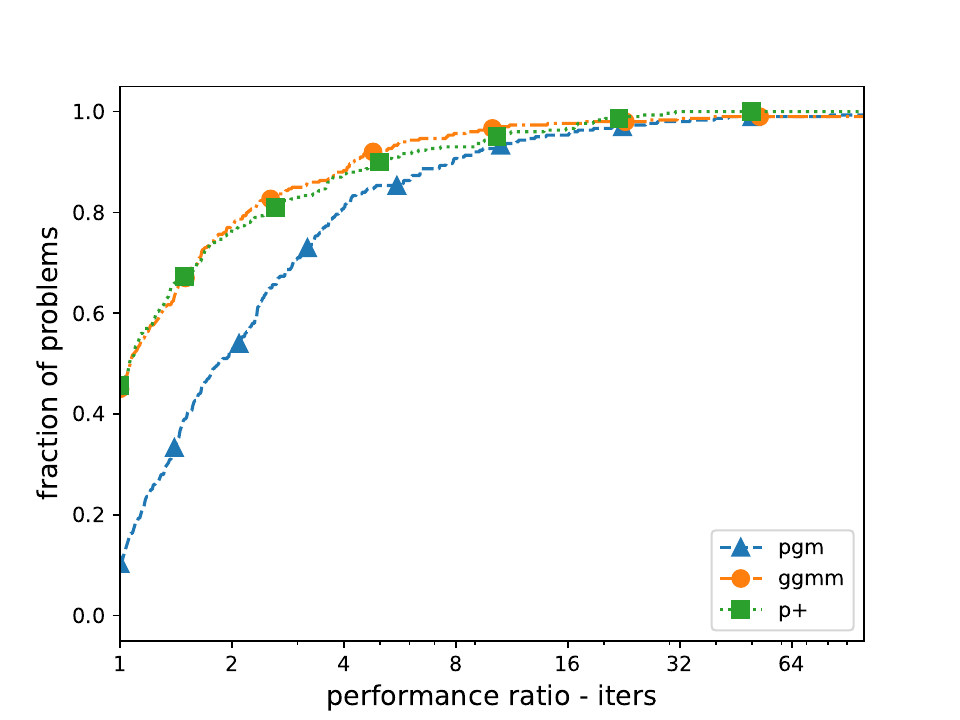}
  \caption{Number of iterations}
  \label{fig:pp_iters_QP}
\end{subfigure}

\begin{subfigure}{0.475\textwidth}
    \centering
    \includegraphics[width=\textwidth]{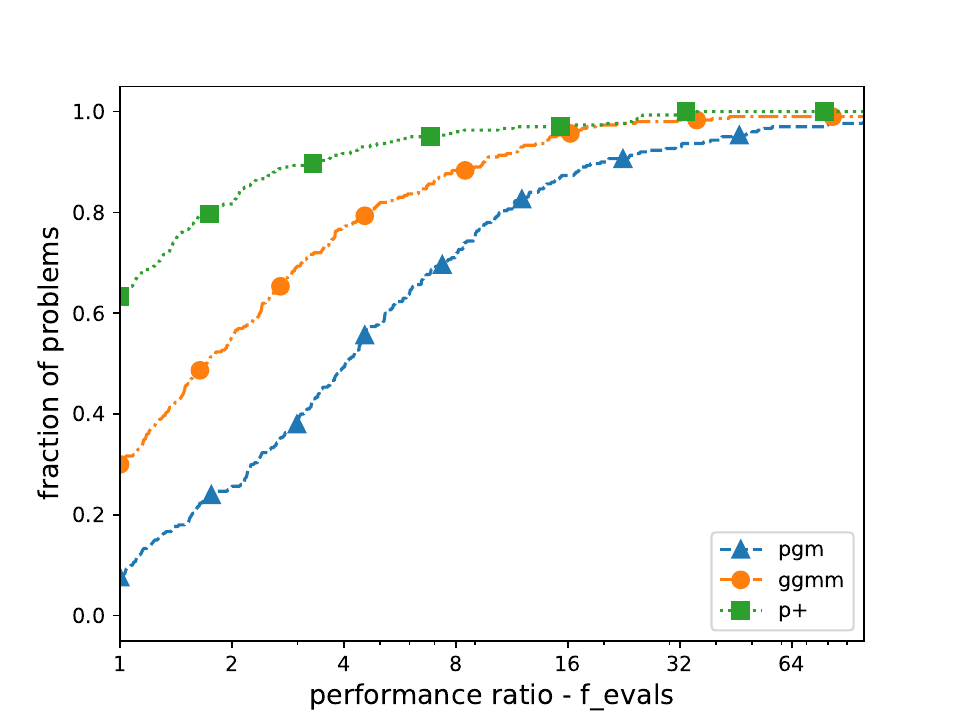}
  \caption{Number of function evaluations}
  \label{fig:pp_f_evals_QP}
    \end{subfigure}    
\hfill
\begin{subfigure}{0.475\textwidth}
    \centering
    \includegraphics[width=\textwidth]{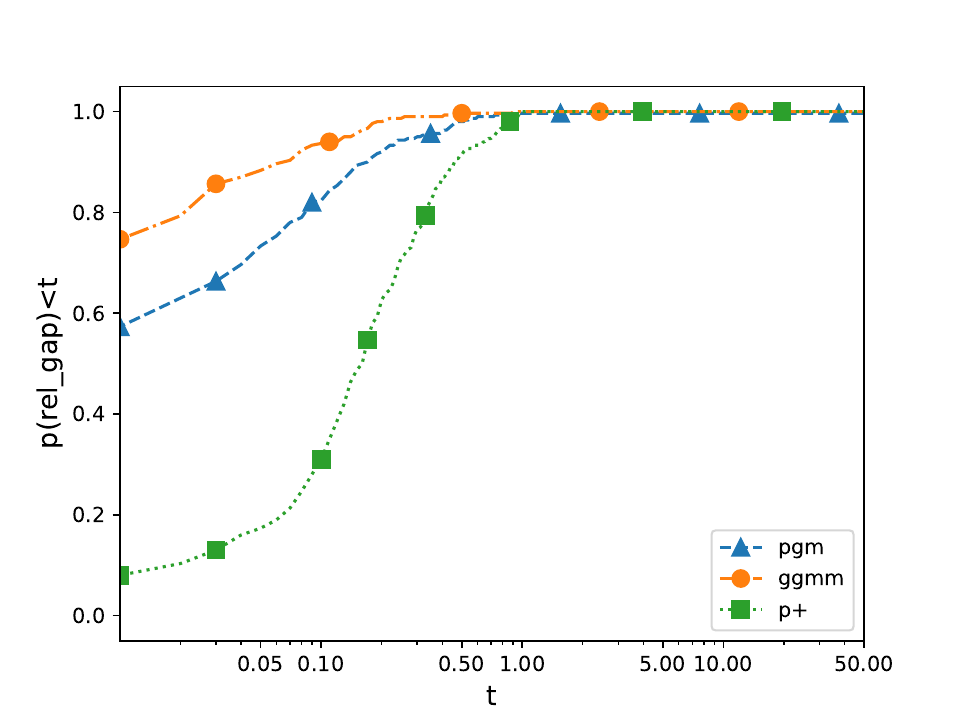}
  \caption{Final function value}
  \label{fig:pp_f_val_QP}
    \end{subfigure}

  \caption{Performance profiles for the quadratic problems.}
  \label{fig:pp_QP}
\end{figure}

\subsubsection{Logistic regression problems}
We now turn to the subclass of logistic regression problems. We consider 24 benchmark problems built from the datasets\footnote{all  available at \href{https://www.csie.ntu.edu.tw/~cjlin/libsvmtools/datasets/binary.html}{\tt https://www.csie.ntu.edu.tw/\~{}cjlin/libsvmtools/datasets/binary.html}} reported in Table \ref{tab:LRproblems}. From this repository, we selected all datasets with a number of variables $n \in [8,100000]$ and at most $100000$ samples. Among the datasets \texttt{a1a}--\texttt{a9a}, we considered only \texttt{a1a} and \texttt{a9a}; similarly, among the datasets \texttt{w1a}--\texttt{w9a}, we considered only \texttt{w1a} and \texttt{w9a}. For each dataset, we solve the cardinality-constrained problem with $s \in \{3,4,5,6,7,8\}$, yielding a total of 144 instances.

The corresponding performance profiles are reported in Figure~\ref{fig:pp_LR}. The time limit is reached on 13 instances by PGM, 14 instances by GGMM, and 17 instances by P+. Moreover, on 6 instances P+ terminates because the steplength parameter satisfies $\gamma_k<10^{-7}$. Overall, GGMM outperforms the other solvers according to all the considered performance metrics. In particular, GGMM achieves the lowest CPU time on $53\%$ of the instances, followed by PGM ($37\%$) and P+ ($3\%$). A similar ranking is observed in terms of the number of iterations and function evaluations. As for the quadratic problems, P+ generally converges to solutions with higher objective values than both PGM and GGMM, while the latter two methods attain comparable objective values.

\begin{table}[h]
\centering
\begin{tabular}{ l l l  l l l }
\hline
Problem & $n$ & $m$ & Problem & $n$ & $m$\\
\hline
\texttt{a1a}  & 123 & 1605 & \texttt{ionosphere}  & 34 & 351  \\

\texttt{a9a}  & 123 & 32561 & \texttt{leukemia}  & 7129 & 38  \\

\texttt{australian}  & 14 & 690 & \texttt{madelon}  & 500 & 2000  \\

\texttt{breast-cancer}  & 10 & 683 & \texttt{mushrooms}  & 112 & 8124   \\

\texttt{cod-rna}  & 8 & 59535 & \texttt{phishing} & 68 & 11055   \\

\texttt{colon-cancer}  & 2000 & 62 & \texttt{rcv1}  & 47236 & 20242   \\

\texttt{diabetes}  & 8 & 768 & \texttt{realsim}  & 20958 & 72309  \\

\texttt{duke}  & 7129 & 44 & \texttt{sonar} & 60 & 208  \\

\texttt{german-number}  & 24 & 1000 &   \texttt{splice} & 60 & 1000  \\

\texttt{gisette}  & 5000 & 6000  & \texttt{svmguide3}  & 21 & 1243 \\

\texttt{heart}  & 13 & 270 & \texttt{w1a} & 300 & 2477  \\

\texttt{ijcnn1}  & 22 & 49990 & \texttt{w8a} & 300 & 49749  \\

\hline
\end{tabular}
\caption{Logistic regression problems: $n$ denotes the number of variables, $m$ the number of samples.}
\label{tab:LRproblems}
\end{table}

\begin{figure}[htbp]
  \centering
  
  \begin{subfigure}{0.475\textwidth}
    \centering
    \includegraphics[width=\textwidth]{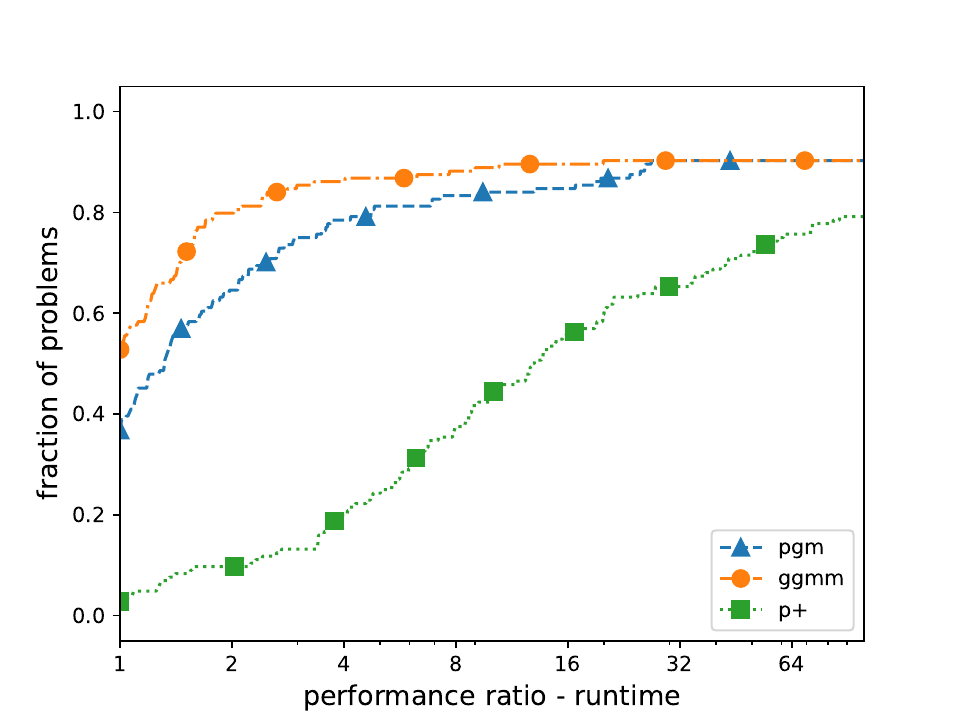}
  \caption{CPU time}
  \label{fig:pp_time_LR}
    \end{subfigure}
\hfill
\begin{subfigure}{0.475\textwidth}
    \centering
    \includegraphics[width=\textwidth]{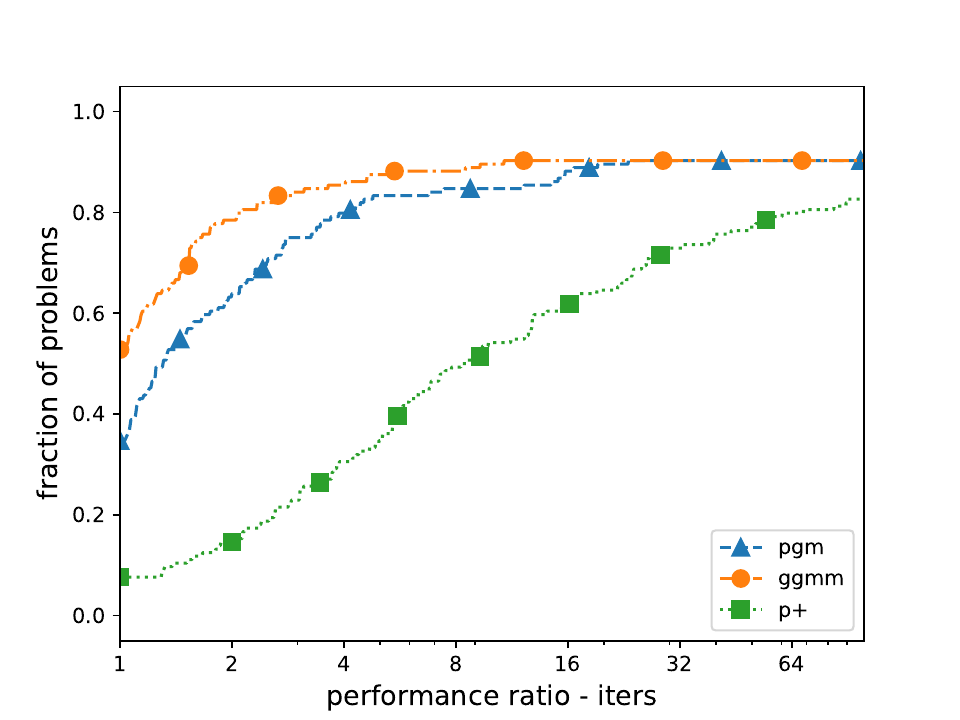}
  \caption{Number of iterations}
  \label{fig:pp_iters_LR}
\end{subfigure}

\begin{subfigure}{0.475\textwidth}
    \centering
    \includegraphics[width=\textwidth]{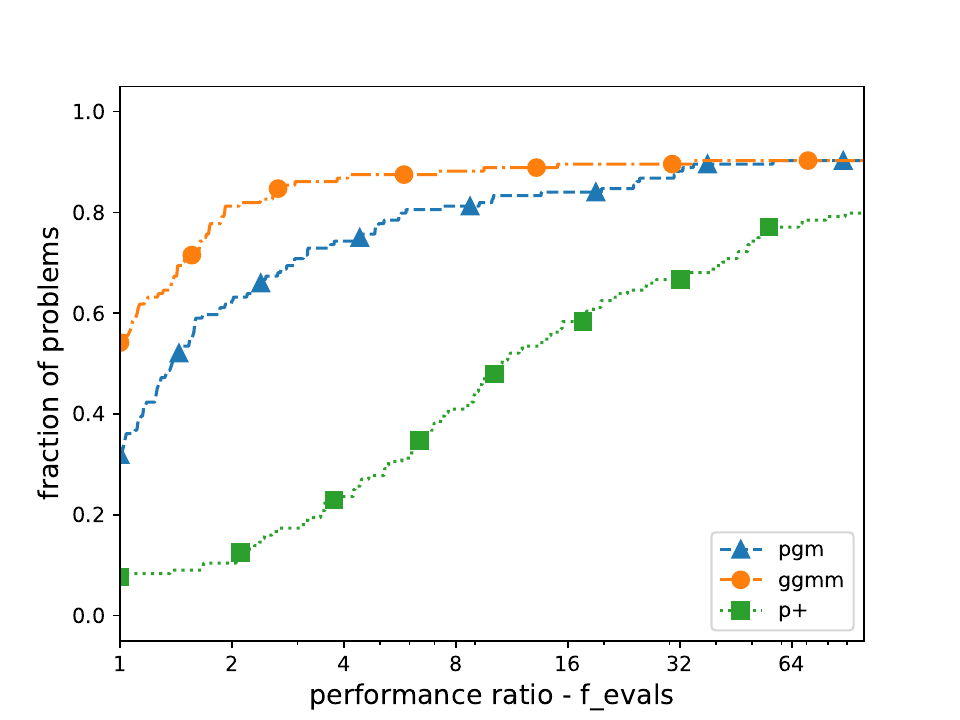}
  \caption{Number of function evaluations}
  \label{fig:pp_f_evals_LR}
    \end{subfigure}
\hfill
    \begin{subfigure}{0.475\textwidth}
    \centering
    \includegraphics[width=\textwidth]{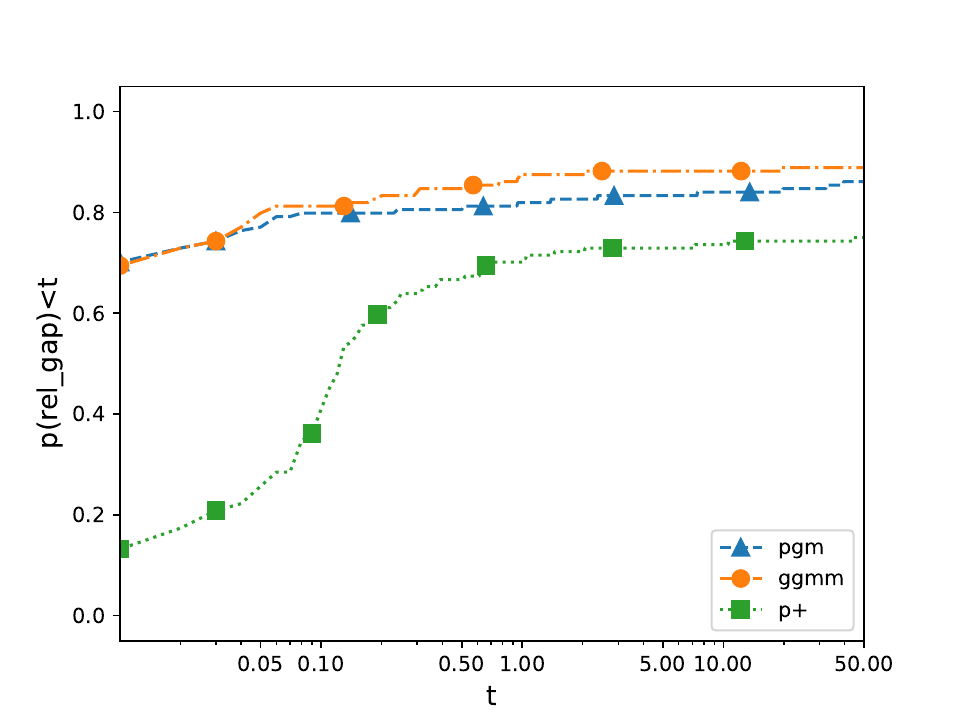}
  \caption{Final function value}
  \label{fig:pp_f_val_LR}
    \end{subfigure}

  \caption{Performance profiles for the logistic regression problems.}
  \label{fig:pp_LR}
\end{figure}

\subsection{Low-rank matrix completion problems}

We now discuss the results of the numerical experiments on the class of low-rank matrix completion problems. We consider problem~\eqref{eqn:prob} with $\mathbb X=\mathbb R^{n \times n}$ and
$$
D=\{X \in \mathbb R^{n \times n} : \text{rank}(X)\le s\},
$$
for some $s \in \{1,\dots,n\}$, and the objective function $f$ given by the loss function of the one-bit low-rank matrix completion problem~\cite{davenport2014matrix}. Specifically, we consider a sparse binary matrix $Y \in \{\pm 1\}^{n \times n}$ and assume that only a small number of its entries are known. Let $\Omega$ be the set of indices $(i,j) \in \{1,\dots,n\}\times \{1,\dots,n\}$ for which the entry $Y_{ij}\in \{\pm 1\}$ is available. The goal is to recover the unknown entries by computing a full matrix $X \in \mathbb R^{n \times n}$ that minimizes the negative log-likelihood over the known entries of $Y$:
$$
f(X)= \frac{1}{|\Omega|} \sum_{(i,j) \in \Omega} \log \big(1+\exp(-Y_{ij}X_{ij})\big).
$$

We consider six values of $n \in \{50,100,200,500,1000,2000\}$, five values of $s\in \{2,3,5,10,20\}$, and five values of $p \in \{1,2,5,10,20\}$. For each combination of $(n,s,p)$, we generate one matrix $M=UV^T \in \mathbb R^{n \times n}$ of rank $s$, where $U,V \in \mathbb R^{n \times s}$ are random matrices. Then, we define $\tilde Y=\text{sign}(M)$ and construct the observed matrix $Y$ by retaining only a fraction $p/100$ of the entries of $\tilde Y$ and masking the remaining ones. This procedure results in a total of 300 instances.

For this class of problems, we consider the following stopping criterion:
$$
\big\|P_{T_D(X)}\big(- \nabla f(X)\big)\big\| < \texttt{tol}.
$$
As presented in~\cite{olikier2025projected}, this condition provides an (approximate) measure of B-stationarity, which, also in this case, coincides with P-stationarity. Here, $P_{T_D(X)}$ denotes the orthogonal projector onto the tangent space of $D$ at the point $X\in D$. Observe that the projection of a matrix $Z$ on $T_D(X)$ can be computed efficiently as
\[
P_{T_D(X)}(Z)=UU^\top Z + Z VV^\top - UU^\top Z VV^\top,
\]
where $U,V \in \mathbb R^{n \times s}$ are the matrices obtained from the SVD of $X = U \Sigma V^T$, where $\Sigma \in \mathbb R^{s \times s}$ is the diagonal matrix of the largest $s$ singular values.

In our experiments, we set $\texttt{tol}=10^{-5}$, and we set the initial point $X_0$ as the null matrix. 

The corresponding performance profiles are reported in Figure~\ref{fig:pp_MC}. All instances are solved by PGM, whereas GGMM hits the time limit on one instance and P+ on 21 instances. GGMM achieves the best CPU time on $93\%$ of the instances, while PGM does so on $7\%$ of the instances and P+ on none. A similar behavior is observed in terms of the number of iterations and function evaluations. Moreover, the final objective values obtained by GGMM are significantly lower than those achieved by PGM and P+.

\begin{figure}[htbp]
  \centering
  
  \begin{subfigure}{0.475\textwidth}
    \centering
    \includegraphics[width=\textwidth]{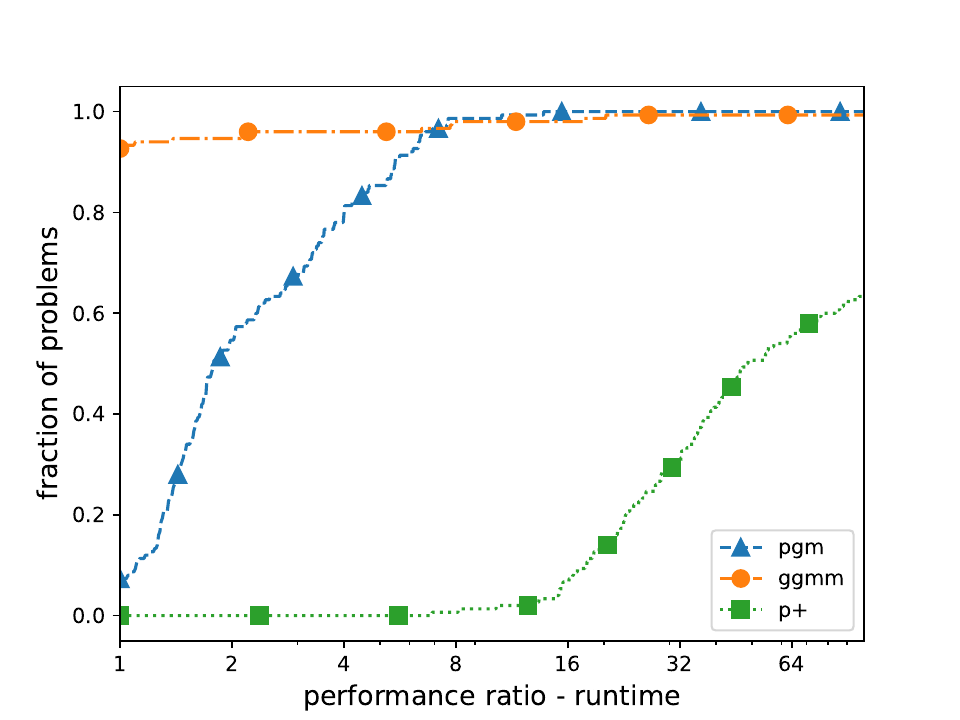}
  \caption{CPU time}
  \label{fig:pp_time_MC}
    \end{subfigure}
\hfill
\begin{subfigure}{0.475\textwidth}
    \centering
    \includegraphics[width=\textwidth]{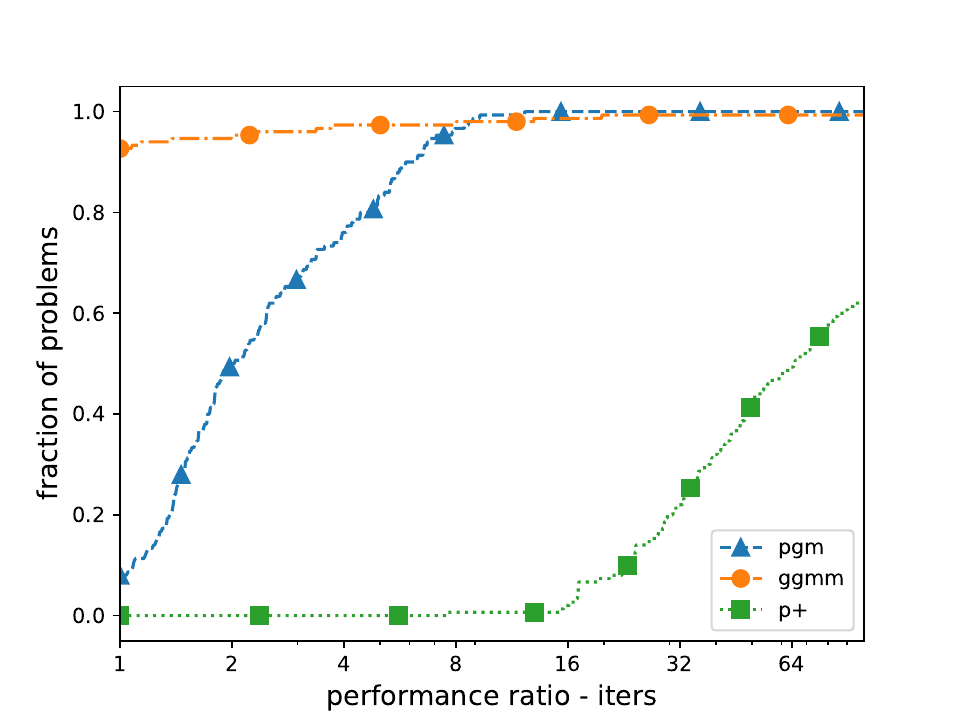}
  \caption{Number of iterations}
  \label{fig:pp_iters_MC}
\end{subfigure}

\begin{subfigure}{0.475\textwidth}
    \centering
    \includegraphics[width=\textwidth]{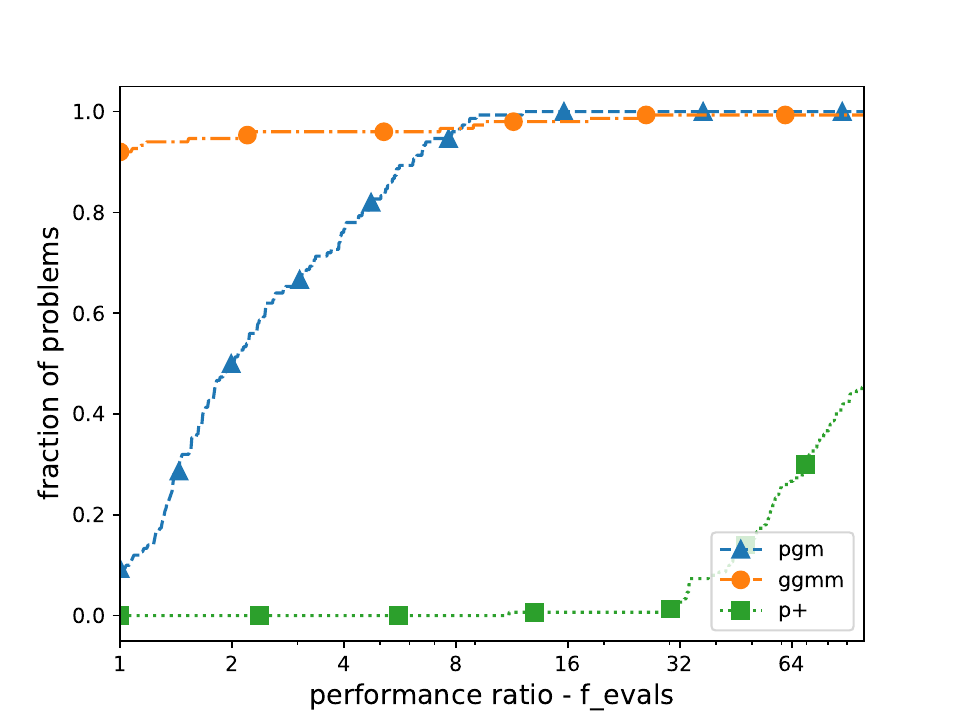}
  \caption{Number of function evaluations}
  \label{fig:pp_f_evals_MC}
    \end{subfigure}
\hfill
    \begin{subfigure}{0.475\textwidth}
    \centering
    \includegraphics[width=\textwidth]{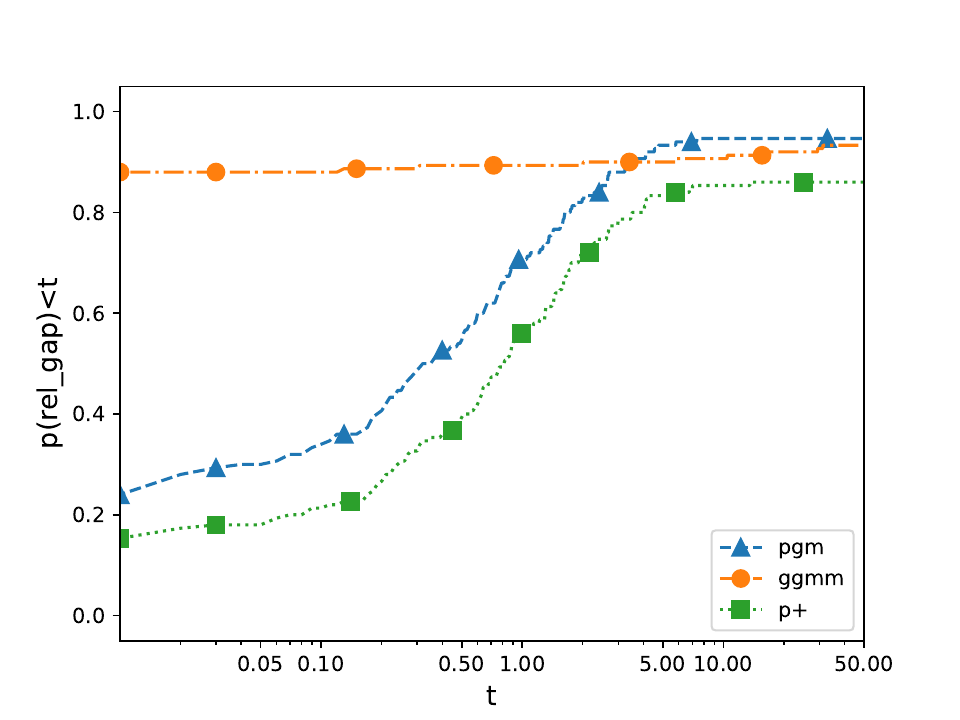}
  \caption{Final function value}
  \label{fig:pp_f_val_MC}
    \end{subfigure}

  \caption{Performance profiles for the one-bit low-rank matrix completion problems.}
  \label{fig:pp_MC}
\end{figure}

In addition to these synthetic instances, we consider a real-world problem involving movie ratings. We use the \texttt{MovieLens-latest-small} dataset\footnote{available at \href{https://grouplens.org/datasets/movielens/}{\tt https://grouplens.org/datasets/movielens/}}, which contains $100836$ ratings (with scores ranging from 1 to 5) assigned by $610$ users to $9724$ movies, represented by a sparse matrix. As in the previous experiments, we consider the one-bit formulation of the problem. Therefore, the ratings are converted into binary labels: $-1$ for scores equal to 1 or 2 (i.e., \textit{dislike}) and $+1$ for scores equal to 3, 4, or 5 (i.e., \textit{like}).

In this experiment, we set the time limit to $3600$ seconds (one hour). GGMM meets the stopping criterion after $908$ seconds and $457$ iterations, whereas PGM and P+ both reach the time limit, after $1970$ and $1608$ iterations, respectively. The evolution of the objective value and of the stopping criterion with respect to time is reported in Figure~\ref{fig:movies}.

\begin{figure}[htbp]
  \centering
  
  \begin{subfigure}{0.475\textwidth}
    \centering
    \includegraphics[width=\textwidth]{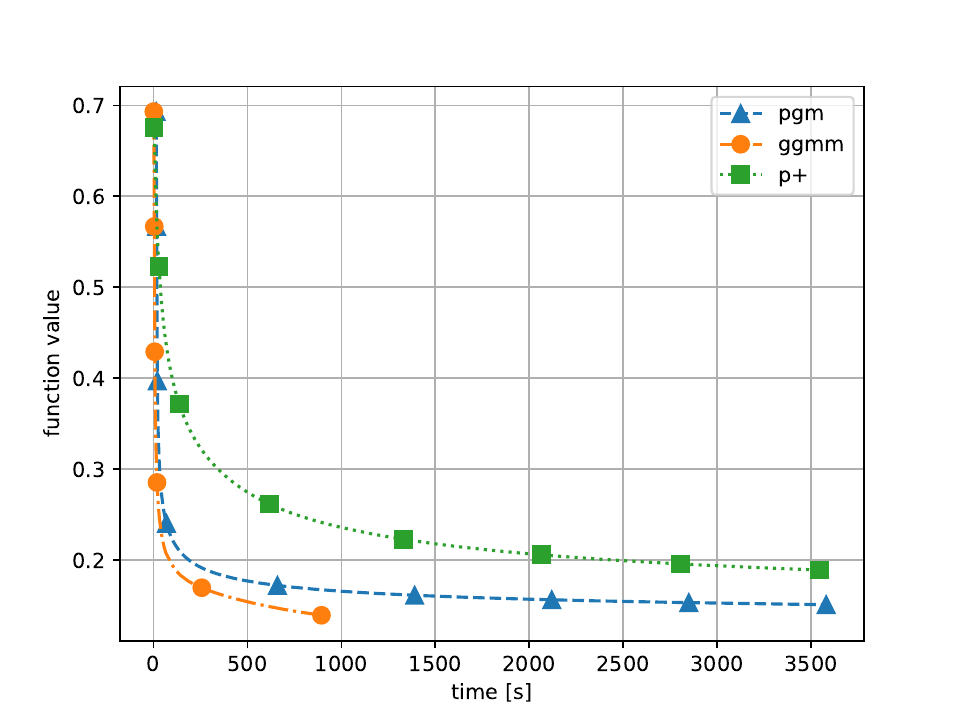}
  \caption{Function value}
  \label{fig:movies_f}
    \end{subfigure}
\hfill
\begin{subfigure}{0.475\textwidth}
    \centering
    \includegraphics[width=\textwidth]{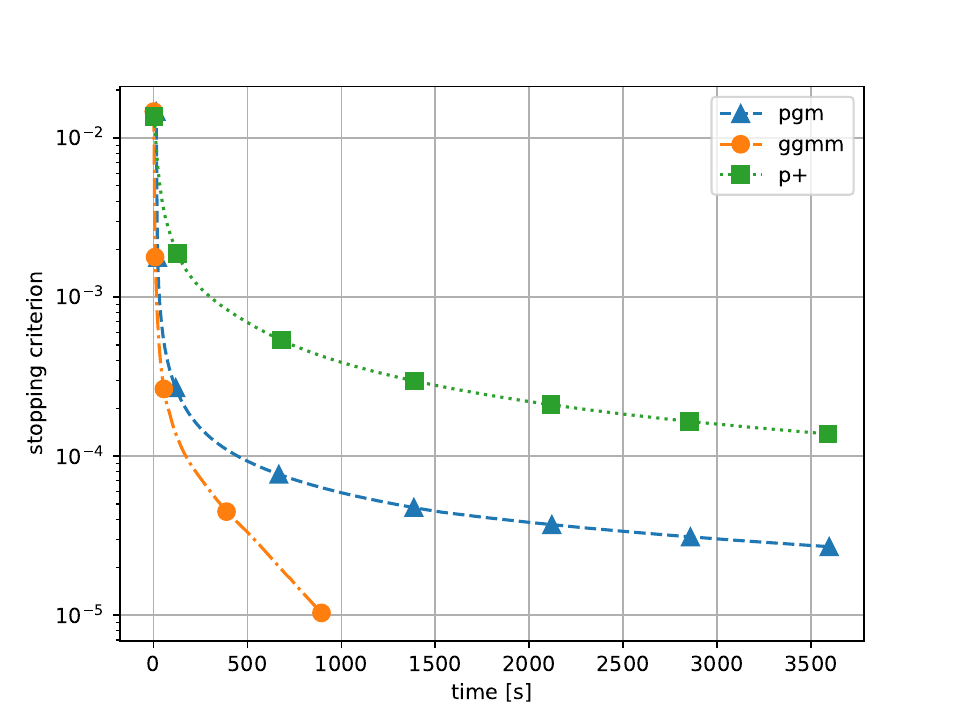}
  \caption{Stopping criterion}
  \label{fig:movies_criterion}
    \end{subfigure}

  \caption{Function value and stopping criterion against time for the \texttt{MovieLens-latest-small} problem.}
  \label{fig:movies}
\end{figure}

\section{Conclusions}
\label{sec:conclusions}
In this paper, we presented a projection-based optimization method that incorporates momentum terms while preserving the strong convergence guarantees of projected gradient methods under nonconvex geometric constraints. The considered algorithmic framework, based on curve backtracking for the pre-projection step,  overcomes the roadblocks associated with straightforward extensions of the PGM and ensures convergence to M-, B- or P-stationary points under assumptions that concern either methodological choices or objective smoothness. These theoretical guarantees match those available from the literature for other frameworks under comparable settings.  

Computational experiments on sparsity- and rank-constrained optimization problems indicate that the proposed approach effectively accelerates convergence while maintaining the robustness and solution quality of existing projection-based framework.

Future work will focus on understanding whether and how the backtracking procedure introduced in this manuscript could be exploited to integrate general descent directions, and not just momentum terms, within the projected gradient framework. Also, the convergence towards B- and P-stationary points of the GGMM algorithm when $\tau$ is set to zero, i.e., when backtracking reverts to negative gradient steps only asymptotically, remains an open question. Finally, the extension of the proposed approach to the more general class of proximal gradient methods for composite optimization could be investigated.

\section*{Declarations}

\subsection*{Funding}

No funding was received for conducting this study.

\subsection*{Competing interests}

The authors have no competing interests to declare that are relevant to the content of this article.

\subsection*{Data Availability Statement}
Data sharing is not applicable to this article as no new data were created or analyzed in this study.

\subsection*{Code Availability Statement}
The code developed for the experimental part of this paper is publicly available at  \href{https://github.com/DiegoScuppa/GGMM}{\tt https://github.com/DiegoScuppa/GGMM}. 

\bibliography{sn-bibliography}

\appendix
\section{Additional Proofs}
\label{sec:appendix}

\begin{proof}[Proof of Proposition~\ref{prop:conv_diff}]
Let $\{x_k\}$ be the infinite sequence generated by Algorithm~\ref{alg:ggmm}. Then, for every $k \in \mathbb N$, since the inner loop terminates, we have
$$f(x_{k+1}) \le \mu_k - \frac{c}{2\alpha_k}\|x_{k+1} - x_k\|^2 \le \mu_k - \frac{c}{2\alpha_{\max}}\|x_{k+1} - x_k\|^2 \le \mu_k.$$
Suppose that the \emph{average-rule} for nonmonotone decrease is employed. As the sequence $\{x_k\}$ has an accumulation point $\bar x \in D$, then the sequences $\{f(x_k)\}$ and $\{\mu_k\}$ converge to the same limit $\bar \mu \in \mathbb R$ (see, e.g.,~\cite[Prop.~4.7]{olikier2025projected}). We therefore have
$$\lim_{k \to \infty} \frac{c}{2\alpha_{\max}}\|x_{k+1} - x_k\|^2 =0;$$
i.e., the thesis. 

Now, suppose that the \emph{max-rule} is employed and that $f$ is bounded from below and uniformly continuous in $\mathcal L_0$. Then, for every $k \in \mathbb N$, let $m_k = \min\{M,k\} $, and let $ l(k)\in\{\max\{0,k-M\},\ldots,k\} $ be an index that attains the maximum of $f$ over the past $m_k$ iterations, i.e., such that
\begin{equation*}
   f(x_{l(k)}) = \max_{\shift= 0, 1, \ldots,m_k} f(x_{k-\shift}) \qquad \forall k \in \mathbb{N}.
\end{equation*}
As the backtracking terminates, we know
\begin{equation}\label{eqn:10}
  f(x_{k+1}) \le f(x_{l(k)}) - \frac{c}{2\alpha_k}\|x_{k+1} - x_k\|^2 .
\end{equation}
We can then point out that $ \{f(x_{l(k)} ) \}_k $ is monotonically decreasing:
noting that $ m_{k+1} \leq m_k + 1 $, we indeed have
\begin{align*}
   f ( x_{l(k+1)} ) & = \max_{\shift=0,1, \ldots, m_{k+1}}f(x_{k+1-\shift}) \\
   & \leq \max_{\shift=0,1, \ldots, m_k +1}f(x_{k+1-\shift}) \\
   & =  \max \big( \max_{\shift = 0, 1, \ldots, m_k}f(x_{k-\shift}),f(x_{k+1}) \big) \\
   & =  \max \big(f(x_{l(k)}),f(x_{k+1}) \big) \\
   & = f(x_{l(k)}),
\end{align*}
where the last equality follows from~\eqref{eqn:10}. Since $ f $ is bounded 
below by some finite value $f^*$, we have
\begin{equation}\label{eqn:Limitfk}
   \lim_{k \to \infty} f( x_{l(k)} ) = f^*.
\end{equation}
Using~\eqref{eqn:10}, replacing $ k $  with $ l(k)-1 $ and rearranging terms, we obtain
\begin{equation*}
   f(x_{l(k)}) -f(x_{l(l(k)-1)}) \leq - \frac{c}{2\alpha_{l(k)-1}}
   \| x_{l(k)} - x_{l(k)-1} \|^2 \leq 0.
\end{equation*}
We can take the limits, recalling~\eqref{eqn:Limitfk}, to obtain
\begin{equation*}
   \lim_{k \to \infty} \frac{\| x_{l(k)} - x_{l(k)-1} \|^2}{\alpha_{l(k)-1}} = 0 .
\end{equation*}
Let us define for notational convenience $ \omega_k = x_{k+1} - x_{k} $; since $ \alpha_k \le \alpha_{\max}$ for all $ k$, we immediately get
\begin{equation}\label{eqn:Ind1}
   \lim_{k \to \infty} \omega_{l(k)-1} = 0.
\end{equation}
By~\eqref{eqn:Limitfk} and~\eqref{eqn:Ind1}, recalling the uniform continuity of $f$, we further get
\begin{equation}\label{eqn:Ind2}
   f^* = \lim_{k \to \infty}f(x_{l(k)}) =
   \lim_{k \to \infty} f(x_{l(k)-1} + \omega_{l(k)-1} ) =
   \lim_{k \to \infty} f (x_{l(k)-1}).
\end{equation}
We now show by induction that 
\begin{equation}\label{eqn:Indj}
   \lim_{k \to \infty} \omega_{l(k)-\shift} = 0 \quad \text{and} \quad 
   \lim_{k \to \infty}f(x_{l(k)-\shift}) = f^* \quad
   \forall \shift \in\N.
\end{equation}
The above limits hold for $\shift=1$, as stated with~\eqref{eqn:Ind1} and~\eqref{eqn:Ind2}. We hence suppose that~\eqref{eqn:Indj} holds for some $ \shift \geq 1 $ and analyze the case of $ \shift+1 $. Replacing $ k $ by  $ l(k)-\shift-1 $ in~\eqref{eqn:10}, and assuming without loss of generality that $ k $ is large enough such that all indices
$ l(k)-\shift-1 $ are positive, we get
\begin{equation*}
  f(x_{l(k)-r}) -f(x_{l(l(k)-r-1)}) \leq - \frac{c}{2\alpha_{l(k)-r-1}}
   \| \omega_{l(k)-r-1} \|^2 .
\end{equation*}
Rearranging and recalling that $ \alpha_k \le
\alpha_{\max} $ for all $ k $, we obtain
\begin{equation*}
   \| \omega_{l(k)-\shift-1} \|^2 \leq \frac{2\alpha_{\max}}{c }
   \big(f(x_{l(l(k)-r-1)})  - f(x_{l(k)-r})\big) .
\end{equation*}
Taking the limit for $ k \to \infty $, exploiting~\eqref{eqn:Limitfk} together with the induction
hypothesis, we conclude that 
\begin{equation}\label{eqn:dstar}
   \lim_{k \to \infty} \omega_{l(k)-\shift-1} = 0,
\end{equation}
i.e., the first limit in~\eqref{eqn:Indj} holds for the case $\shift+1$.
The second limit in~\eqref{eqn:Indj} instead is obtained using the first one and the uniform continuity of $f$, that allow us to write from 
\begin{equation*}
   \lim_{k \to \infty} f(x_{l(k)-(\shift+1)}) =
   \lim_{k \to \infty} f(x_{l(k)-(\shift+1)} + \omega_{l(k)-(\shift+1)}) =
   \lim_{k \to \infty} f(x_{l(k)-\shift}) =
   f^*.
\end{equation*}
We finally have to prove that $ \lim_{k \to \infty} \omega_k = 0 $. 
Let us assume, by contradiction, that the limit is false, i.e.,  there exists constant $ \rho > 0 $ and a (suitably shifted, for notational
simplicity) subsequence $K \subset \mathbb N $
such that
\begin{equation}\label{eqn:Contrad}
   \| \omega_{k-M-1} \| \geq \rho \qquad \forall k \in K.
\end{equation}
For each $ k \in K $, the corresponding index $ l(k) $ is found by definition among
$ k - M, k - M + 1, \ldots, k $. We can therefore write $ k - M - 1 = l(k) - \shift_k $
for some index $ \shift_k \in \{ 1, 2, \ldots, M+1 \} $. Since the number of possible indices $ \shift_k $ is finite, we can assume without loss of generality that
$ \shift_k = \shift $ for all $k$. By~\eqref{eqn:Indj} we then obtain
\begin{equation*}
   \lim_{k \to \infty, k \in K} \omega_{k-M-1} = \lim_{k \to \infty, k \in K} \omega_{l(k) - \shift} = 0,
\end{equation*}
which contradicts~\eqref{eqn:Contrad}.
\end{proof}

\medskip
\begin{proof}[Proof of Proposition~\ref{prop:AcceptedStep}]
Define $\kappa$ as in Proposition~\ref{prop:ArbitrarySmallAlpha}. Let $\xi \in (0, \infty)$ be small enough to ensure
\begin{subequations}
\begin{align}
\label{eqn:def_alpha_x_small1}
\sup_{y \in  B\left(\bar x, \frac{7\xi}{2\delta}\|\nabla f(\bar x)\|\right) \cap D \setminus \{\bar x\}}
\frac{|f(y)-f(\bar x) - \langle\nabla f(\bar x),y-\bar x\rangle|}{\|y-\bar x\|}
&< \frac{(1-c)\sqrt{1-\kappa^2}\|\nabla f(\bar x)\|}{4\left(1 + \frac{8}{3(1-\kappa)}\right)},\\
\label{eqn:def_alpha_x_small2}
\sup_{y \in  B\left(\bar x, \frac{7\xi}{2\delta}\|\nabla f(\bar x)\|\right) \cap D}
\|\nabla f(y)-\nabla f(\bar x)\|
&< \frac{(1-c)\sqrt{1-\kappa^2}}{4} \|\nabla f(\bar x)\|.
\end{align}
\end{subequations}
Such a value $\xi$ certainly exists by the definition and the continuity of the gradient at $\bar{x}$.

Then, we can define $\varepsilon \coloneq \min\left\{ \xi,\delta\tau\right\}$, let $\alpha_{\bar x} \in (0, \varepsilon]$ and choose $\bar{\rho}(\alpha_{\bar x})\in(0,\infty)$ as given in Proposition~\ref{prop:ArbitrarySmallAlpha}. 

Since $\alpha_{\bar x} \le \varepsilon \le \delta\tau$, we can consequently observe that, for every $\alpha \in [\alpha_{\bar x}, \alpha_{\bar x}/\delta]$, we have $\alpha \le \alpha_{\bar x}/\delta\le\tau$, and hence $\max\{\alpha-\tau,0\}=0$.   For all $\alpha_0\in[\alpha_{\min}, \alpha_{\max}]$, $\beta_0\in[0,\beta_{\max}]$, and $a\ge1$, we therefore have 
$\beta(\alpha) = \beta_0 \big(\frac{\max\{\alpha-\tau,0\}}{\alpha_0-\tau}\big)^a =  0$. 
Then, for all $s \in \mathbb X$,  $P_D(x-\alpha\nabla f(x)+\beta(\alpha) s)=P_D(x-\alpha\nabla f(x))$. 

We define
$
\rho = \min\left\{\bar{\rho}(\alpha_{\bar x}), \alpha_{\bar x} \|\nabla f(\bar x)\|\right\},
$
and note, for all $x \in B(\bar x, \rho) \cap D$, that
\begin{equation*}
\|x-\bar x\|
< \rho \le \alpha_{\bar x} \|\nabla f(\bar x)\|
< \frac{7\alpha_{\bar x}}{2\delta} \|\nabla f(\bar x)\|
\leq \frac{7\xi }{2\delta} \|\nabla f(\bar x)\|,
\end{equation*}
so that from~\eqref{eqn:def_alpha_x_small2} we have $\|\nabla f(x)-\nabla f(\bar x)\| < \frac{\|\nabla f(\bar x)\|}{4}$ and then
\begin{equation}
    \label{eqn:nablas_ushortx_x}
    \begin{aligned}
        \frac{3}{4} \|\nabla f(\bar x)\|
& < \|\nabla f(\bar x)\|-\|\nabla f(x)-\nabla f(\bar x)\|  \le \|\nabla f(x)\| \\
& \le \|\nabla f(\bar x)\|+\|\nabla f(x)-\nabla f(\bar x)\|  < \frac{5}{4} \|\nabla f(\bar x)\|.
    \end{aligned}
\end{equation}
Thus, for all $x \in B(\bar x, \rho) \cap D$, $\alpha \in [\alpha_{\bar x}, \alpha_{\bar x}/\delta]$, and $y \in P_D(x-\alpha\nabla f(x)+\beta(\alpha) s)=P_D(x-\alpha\nabla f(x))$, we can write
\begin{equation}
    \begin{aligned}
f(y)
&= f(x) + \langle \nabla f(\bar x),y-x\rangle  + \left(f(\bar x) - f(x) - \langle\nabla f(\bar x),\bar x-x\rangle\right) \\
&\qquad + \left(f(y) - f(\bar x) - \langle\nabla f(\bar x),y-\bar x\rangle\right) \\
&\le f(x) + \langle\nabla f(\bar x),y-x\rangle +
\frac{(1-c)\sqrt{1-\kappa^2}\|\nabla f(\bar x)\|}{4\left(1 + \frac{8}{3(1-\kappa)}\right)} \left(\|\bar x-x\| + \|y-\bar x\|\right),
\label{eqn:UpperBoundfDiffGrad}
\end{aligned}
\end{equation}
where the inequality follows from~\eqref{eqn:def_alpha_x_small1}. 
Observe that
\begin{align*}
\|y-\bar x\|
&\le \|y-x\| + \|x-\bar x\| \\
&\le 2 \alpha \|\nabla f(x)\| + \rho \\
&\le \frac{2 \alpha_{\bar x}}{\delta} \|\nabla f(x)\|
+ \alpha_{\bar x} \|\nabla f(\bar x)\| \\
&< \frac{5\alpha_{\bar x}}{2\delta} \|\nabla f(\bar x)\|
+ \alpha_{\bar x} \|\nabla f(\bar x)\|  \\
&\le \frac{7\alpha_{\bar x}}{2\delta} \|\nabla f(\bar x)\|,
\end{align*}
where the second inequality follows from~\eqref{eqn:propr_proj1}, and the forth one from~\eqref{eqn:nablas_ushortx_x}.

Continuing from~\eqref{eqn:UpperBoundfDiffGrad}, we have:
\begin{align*}
f(y)
& \le f(x) + \langle\nabla f(\bar x), y-x\rangle +
\frac{(1-c)\sqrt{1-\kappa^2}\|\nabla f(\bar x)\|}{
4\left(1 + \frac{8}{3(1-\kappa)}\right)} \left(2\|\bar x-x\| + \|y-x\|\right) \\
&< f(x) + \langle\nabla f(\bar x),y-x\rangle
+ \frac{(1-c)\sqrt{1-\kappa^2}\|\nabla f(\bar x)\|}{4} \|y-x\| \\
&\le f(x) + \langle\nabla f(x), y-x\rangle + \|\nabla f(\bar x)-\nabla f(x)\| \|y-x\| \\
&\qquad + \frac{(1-c)\sqrt{1-\kappa^2}\|\nabla f(\bar x)\|}{4} \|y-x\| \\
&\le f(x) + \langle\nabla f(x),y-x\rangle + \frac{(1-c)\sqrt{1-\kappa^2}\|\nabla f(\bar x)\|}{2} \|y-x\|\\
&< f(x) + \langle\nabla f(x),y-x\rangle + (1-c)\sqrt{1-\kappa^2} \|\nabla f(x)\|\|y-x\|\\
&\le f(x) + \langle\nabla f(x),y-x\rangle - (1-c) \langle\nabla f(x),y-x\rangle\\
&= f(x) + c \langle \nabla f(x), y-x\rangle \le f(x) - \frac{c}{2\alpha}\|y-x\|^2,
\end{align*}
where the fourth inequality comes  from~\eqref{eqn:def_alpha_x_small2}, the fifth one follows from~\eqref{eqn:nablas_ushortx_x}, the sixth one from Proposition~\ref{prop:ArbitrarySmallAlpha} and the last one from~\eqref{eqn:propr_proj2}. As for the second inequality, it is true because
\begin{align*}
\|y-x\|
&\ge \alpha \|\nabla f(x)\| - \|x-\alpha \nabla f(x)-y\|
\\
&= \alpha \|\nabla f(x)\| - d_D(x-\alpha\nabla f(x)) \\
&\ge (1-\kappa) \alpha \|\nabla f(x)\|
\\
&\ge \frac{3}{4} (1-\kappa) \rho
\\
&> \frac{3}{4} (1-\kappa) \|x-\bar x\|,
\end{align*}
where in the first step we exploit triangle inequality, the third follows from Proposition~\ref{prop:ArbitrarySmallAlpha} and the fourth inequality holds by~\eqref{eqn:nablas_ushortx_x}, the definition of $\rho$, and $\alpha_{\bar x} \le \alpha$.
\end{proof}

\end{document}